\documentclass[12pt]{article}
\usepackage{geometry}
\usepackage{amsmath,amssymb,amsthm}
\usepackage{pstricks}
\usepackage{graphicx}
\usepackage{pst-fill,pst-grad}
\usepackage{pst-plot}
\usepackage{hyperref}
\usepackage{epsfig}
\usepackage{bbm}
\usepackage[latin1]{inputenc}
\usepackage{enumerate}

\newtheorem{theo}{Theorem}[section]
\newtheorem{lem}{Lemma}[section]
\newtheorem{maintheo}{Main Theorem}
\newtheorem{defi}{Definition}
\newtheorem{cor}{Corollary}[section]
\newtheorem{prop}{Proposition}[section]
\newtheorem{rmk}{Remark}
\newcommand{\eps}{\varepsilon}

\newcommand{\R}{\mathbb{R}}

\newcommand{\T}{\mathbb{T}^1}

\renewcommand{\eps}{\varepsilon}

\renewcommand{\a}{\alpha}

\newcommand{\D}{\Delta}
\numberwithin{equation}{section}
\begin{document}
\title{\textbf{TRAVELING WAVE SOLUTIONS OF
ADVECTION-DIFFUSION EQUATIONS WITH NONLINEAR DIFFUSION}}
\vspace{1cm}
\author{
L. Monsaingeon
\footnote{Institut de Math\'ematiques de Toulouse, Universit\'e Paul Sabatier, \href{mailto:leonard.monsaingeon@math.univ-toulouse.fr}{\nolinkurl{leonard.monsaingeon@math.univ-toulouse.fr}}
},
 A. Novikov
\footnote{Penn State University, \href{mailto:anovikov@math.psu.edu}{\nolinkurl{anovikov@math.psu.edu}}},
 J.M. Roquejoffre
\footnote{Institut de Math\'ematiques de Toulouse, Universit\'e Paul Sabatier, \href{mailto:jean-michel.roquejoffre@math.univ-toulouse.fr}{\nolinkurl{jean-michel.roquejoffre@math.univ-toulouse.fr}}}}
\maketitle

\vspace{1cm}

\begin{abstract}
We study the existence of particular traveling wave solutions  of a nonlinear parabolic degenerate diffusion equation with a shear flow. 
Under some assumptions we prove that such solutions exist at least for propagation 
speeds $c\in]c_*,+\infty[$, where $c_*>0$ is explicitly computed but may not be optimal. We also prove that a free boundary hypersurface separates a region where $u=0$ and a region where $u>0$, and that this free boundary can be globally parametrized as a Lipschitz continuous graph under some additional non-degeneracy hypothesis; we investigate solutions which are, in the region $u>0$, planar and linear at infinity in the propagation direction, with slope equal to the propagation speed.
\end{abstract}

\vspace{1cm}

\section{Introduction}
Consider the advection-diffusion equation
\begin{equation}
\partial_t T-\nabla\cdot(\lambda \nabla T)+ \nabla \cdot \left( V T \right)=0, \qquad(t,X)\in \R^+\times \R^d
\label{eq:physical_model}
\end{equation}
where $T\geq 0$ is temperature, $\lambda\geq 0$ is a diffusion coefficient and $V=V(x_1,...,x_d)\in\R^d$ is a prescribed flow. In the context of high temperature hydrodynamics, the diffusion coefficient $\lambda$ cannot be assumed to be constant as for the usual heat equation, but rather of the form $\lambda=\lambda(T)=\lambda_0 T^{m}$ for some conductivity exponent $m>0$ depending on the model, see \cite{ZeldRaizer-physics}. We will consider here the case $m\neq 1$. In Physics of Plasmas and particularly in the context of Inertial Confinement Fusion, the dominant mechanism of heat transfer is the so-called electronic Spitzer heat diffusivity, corresponding to $m=5/2$ in the formula above.
\par
Suitably rescaling one may set $\lambda_0=m+1$, yielding the nonlinear parabolic equation
\begin{equation}
\partial_t T -\Delta\left(T^{m+1}\right) + \nabla\cdot(VT)=0.
\label{eq:PMED}
\end{equation}
When temperature takes negligible values, say $T =\eps \to 0$, then the diffusion coefficient $\lambda(T)=\lambda_0T^m$ may vanish and the equation becomes degenerate. As a result free boundaries may arise. We are interested here in traveling waves with such free boundaries $\Gamma=\partial\{T>0\}\neq\emptyset$, and in addition $T \to +\infty$ in the propagation direction.
\\
\par
When $V\equiv 0$ \eqref{eq:PMED} is usually called the Porous Medium Equation - PME in short -
\begin{equation}
\partial_t T -\Delta\left(T^{m+1}\right)=0
\tag{PME}
\label{eq:truePME}
\end{equation}
and has been widely studied in the literature. We refer the reader to the book~\cite{Vazq-PME} for general references on this topic and to~\cite{AronsonBenilan-regularite,AronsonCaff-initialtrace,BenilanPierre-solutionsPMERn} for well-posedness of the Cauchy problem and regularity questions. As for most of the free boundary scenarios, we do not expect smooth solutions to exist, since along the free boundary a gradient discontinuity may occur: a main difficulty is to develop a suitable notion of viscosity and/or weak solutions; see \cite{CrandallIshiiLions-userguide} for a general theory of viscosity solutions and \cite{CaffVazq-viscPME} in the particular case of the PME, \cite{Vazq-PME} for weak solutions.
\par
The question of parametrization, time evolution and regularity of the free boundary for \eqref{eq:truePME} is not trivial. It has been studied in detail in~\cite{CaffFried-regularity,CaffarelliVazquezWolanski-lipschitzPME,CaffWol-C1alpha}. 
 When the flow is potential $V=\nabla \Phi$ \eqref{eq:PMED} has recently been studied in \cite{Kim}, where the authors investigate the long time asymptotics of the free boundary for compactly supported solutions.
\\
\par
We consider here a two-dimensional periodic incompressible shear flow
$$
V(x,y)=\left(\begin{array}{c}
		\alpha(y)\\
		0
             \end{array}
\right), \qquad \alpha(y+1)=\alpha(y)~
$$
for a sufficiently smooth $\alpha(y)$, which we normalize to be mean-zero $\int\limits_0^1\alpha(y)\mathrm{d}y=0$. In this setting \eqref{eq:PMED} becomes the the following advection-diffusion equation
\begin{equation}
\partial_t T-\D(T^{m+1})+\a(y)\partial_xT=0
\tag{AD-E}
\label{eq:PME0}
\end{equation}
with $1$-periodic boundary conditions in the $y$ direction. 
\\
\par
For physically relevant temperature $T\geq0$ it is standard to use the so-called pressure variable $u=\frac{m+1}{m}T^{m}$, which satisfies
\begin{equation}
\partial_t u-m u\D u+\a(y)\partial_xu=|\nabla u|^2.
\label{eq:PMEparabolic}
\end{equation}
\begin{rmk}
When $m=1$ pressure $u$ is proportional to temperature $T$, and this particular case will not be considered here.
\end{rmk}
Looking for wave solutions  $u(t,x,y)=p(x+ct,y)$ yields the following equation for the wave profiles $p(x,y)$
\begin{equation}
-mp\D p+(c+\a)p_x=|\nabla p|^2,\qquad (x,y)\in\R\times\T.
\label{eq:PME}
\end{equation}
In the case of a trivial flow $\a\equiv 0$ it is well-known~\cite{Vazq-PME} that for any prescribed propagation speed $c>0$ there exists a corresponding planar viscosity solution given by
\begin{equation}\label{linDrift}
p(x,y)=p_c(x)=c[x]^+,
\end{equation}
where $[.]^+$ denotes the positive part. This profile is trivial for $x\leq 0$ and linear for $x>0$, with slope exactly equal to the speed $c$. The free boundary $\Gamma=\{x=0\}$ moves in the original frame with constant speed $x(t)=-ct+cst$, and the slope at infinity therefore fully determines the propagation.
\par
In this particular case the free boundary is non-degenerate, $\nabla p=(c,0)\neq 0$ in the ``hot'' region $p>0$. The differential equation satisfied by the free boundary $\Gamma$ in the general case was specified in~\cite{CaffWol-C1alpha},
 where  the authors also show that if the initial free boundary is non-degenerate then it starts to move immediately with normal velocity $V=-\left.\nabla p\right|_{\Gamma}$.
\par
In presence of a nontrivial flow $\alpha\neq 0$ a natural question to ask is whether \eqref{eq:PME0} can be considered as a perturbation of \eqref{eq:truePME}.   More specifically we are interested here in the following 
questions: (i) do $y$-periodic traveling waves behaving linearly at infinity and possessing free boundaries still exist? (ii) If so for which propagation speeds $c>0$, and is it still true that the slope at infinity equals the speed $c$? (iii) Is the free boundary wrinkled by the flow and how can we parametrize it? (iv) Is the free boundary non-degenerate and what is its regularity?
\par
We answer the first three questions above, at least for propagation speeds large enough.
\begin{maintheo}
Let $c_*:=-\min \alpha>0$: for any $c> c_*$ there exists a nontrivial traveling wave profile, which is a continuous viscosity solution $p(x,y)\geq 0$ of \eqref{eq:PME} on the infinite cylinder. This profile satisfies
\begin{enumerate}
\item if $D_+:=\{p>0\}$ denotes the positive set, we have $D^+\neq\emptyset$ and $\left.p\right|_{D^+}\in\mathcal{C}^{\infty}(D_+)$,
\item $p$ is globally Lipschitz,
\item $p$ is planar and linear $p(x,y)\sim cx$ uniformly in $y$ when $x\rightarrow +\infty$.
\end{enumerate}
Moreover there exists a free boundary $\Gamma=\partial (D^+)\neq\emptyset$ which can be parametrized as follows: there exists an upper semi-continuous function $I(y)$ such that $p(x,y)>0\Leftrightarrow x>I(y)$. Further:
\begin{itemize}
 \item 
If $y_0$ is a continuity point of $I$ then $\Gamma\cap\{y=y_0\}=(I(y_0),y_0)$. 
\item
If $y_0$ is a discontinuity point then $\Gamma\cap\{y=y_0\}=[\underline{I}(y_0),I(y_0)]\times\{y=y_0\}$, where $\underline{I}(y_0):=\displaystyle{\liminf_{y\rightarrow y_0}} I(y)$.
\end{itemize}
\label{maintheo:main}
\end{maintheo}
The question 4 is sill open.  The non-degeneracy of the pressure at the free boundary $\left.\nabla p\right|_{\Gamma}\neq 0$ and the free boundary regularity are closely related. For the PME it was discussed in~\cite{CaffFried-regularity,CaffarelliVazquezWolanski-lipschitzPME,CaffWol-C1alpha}. We have, however, a partial answer under some strong non-degeneracy assumption.
\begin{prop}
With the same hypotheses as in Theorem~\ref{maintheo:main}, assume that the non-degeneracy condition $p_x\geq a>0$ holds in $D^+$ for some constant $a$. Then the function $I(y)$ defined in \ref{maintheo:main} is Lipschitz, and $\Gamma=\partial\{p>0\}=\Big{\{}(x,y),\quad x=I(y)\Big{\}}$.
\label{prop:lipschitz}
\end{prop}
\begin{rmk}
The condition of linear growth at infinity is natural because it mimics the planar traveling wave~\eqref{linDrift} for the PME. Let us also point out, a posteriori, that this linearity appears very naturally in our proof, see Section~\ref{section:linear}.
\end{rmk}
Recalling that we normalized $\int\limits_{\T}\alpha(y)dy=0$ ,we will always assume in the following that the propagation speed $c>0$ is large enough such that
 \begin{equation}
 0<c_0\leq c+\a\leq c_1
\label{hyp:c0<c+a<c1}
\end{equation}
for some constants $c_0,c_1$. This is indeed consistent with $c>c_*=-\min \alpha>0$ in the main Theorem \ref{maintheo:main}.
\\
\par
The method of proof of Theorem~\ref{maintheo:main} is standard. We refer the reader to \cite{BerestyckiCaffarelliNirenberg-uniformEstimatesFB} for a general review of this method and to \cite{CaffFried-regularity} for the special case of the PME. The proof relies on a simple observation: if $p\geq \delta>0$ \eqref{eq:PME} is uniformly elliptic; we shall refer in the sequel to any solution $p\geq\delta>0$ as a $\delta$-solution. The main steps are the following.
\par
We first regularize~\eqref{eq:PME}  by considering its $\delta$-solutions with $\delta \ll 1$ on finite cylinders 
$[-L,L]\times\T$, $L \gg 1$ with suitable boundary conditions. In Section~\ref{section:finitedomain} we solve this regularized uniformly elliptic problem, and derive monotonicity estimates of its solution in the $x$ direction. In Section \ref{section:unif_infinite_domain} we take the limit $L\rightarrow +\infty$ for fixed $\delta>0$, and establish
\begin{theo}
For any $\delta>0$ small enough there exists a smooth $\delta$-solution $p\geq\delta$ on the infinite cylinder such that $\displaystyle{\lim_{x\rightarrow -\infty}} \;p(x,y)=\delta$ and $p(x,y)\underset{x\to+\infty}{\sim} cx$ uniformly in $y$.
\label{theo:exists_delta_solutions}
\end{theo}
\noindent
The linear behavior will be actually proved in section \ref{section:linear} for the final viscosity solution, but the proof can be easily adapted for the $\delta$-solutions. We complete the proof of parts 1\mbox{.} and 2\mbox{.} of Theorem~\ref{maintheo:main} in Section~\ref{section:interface} by taking the degenerate limit $\delta\rightarrow 0^+$. Section~\ref{section:interface} also contains the analysis of the free boundary and the proof of Proposition~\ref{prop:lipschitz}. In Section~\ref{section:linear} we investigate the linear growth and planar behavior at infinity. We refine part 3\mbox{.} of Theorem~\ref{maintheo:main} as
\begin{theo}
Both for the viscosity solution of Main Theorem \ref{maintheo:main} and the $\delta$-solution of Theorem \ref{theo:exists_delta_solutions}, the following holds when $x\to +\infty$:
\begin{enumerate}
 \item 
$p(x,y)\sim cx$, $p_x(x,y)\sim c$ and $p_y(x,y)\rightarrow 0$ uniformly in $y$
\item
If $1<m\notin\mathbb{N}$ and $N:=[m]$, there exist $q_1...q_N$ and $q^*\in\R$ such that
$$
p(x,y)=cx+x\left(q_1x^{-\frac{1}{m}}+...+q_Nx^{-\frac{N}{m}}\right)+q^*+o(1)
$$
\end{enumerate}
\end{theo}
\noindent
The second part is novel compared to the PME, for which the planar wave is exactly linear $p=cx$ at infinity. Once again, we prove this statement for the final viscosity solution, but the proof extends to the $\delta$-solutions. Section~\ref{section:uniqueness} is finally devoted to uniqueness of the wave profiles, and we establish
\begin{theo}
The $\delta$-solutions of Theorem \ref{theo:exists_delta_solutions} are unique up to $x$-translations.
\end{theo}
%
%
\section{$\delta$-solutions on finite domains}
\label{section:finitedomain}
Here we solve~\eqref{eq:PME} on truncated cylinders $D_L=[-L,L]\times \T$, $L\gg 1$ with a uniform ellipticity condition $p\geq \delta>0$.  We show in this section that this ellipticity can be achieved by
setting suitable boundary conditions, and we therefore consider the following problem
\begin{equation}
0<\delta<A<B,\qquad 
\left\{
\begin{array}{cc}
 -mp\D p+(c+\a)p_x=|\nabla p|^2 & \left(D_L\right),\\
 p=A,  & (x=-L),\\
 p=B,  & (x=+L),
\end{array}
\right. 
\label{eq:DLproblem}
\end{equation}
where the constants $A$ and $B$ are specified later. 

We will show that any solution of~\eqref{eq:DLproblem} must satisfy $p_x> 0$, and therefore $p\geq A>0$ on $D_L$. Thus~\eqref{eq:DLproblem} is uniformly elliptic. We prove this $x$-monotonicity of $p$ using the following non-linear comparison principle, which relies on the celebrated Sliding Method of Berestycki and Nirenberg~\cite{BeresNiren-sliding}.
\\
\par
Let $a<b$, $\Omega=]a,b[\times\T$ and for any function $f\in\mathcal{C}^2(\Omega)\cap\mathcal{C}(\overline{\Omega})$ define the nonlinear differential operator
\begin{equation}
\Phi(f):=-mf\D f+(c+\a)f_x-|\nabla f|^2.
\label{eq:defNLPhi}
\end{equation}
\begin{theo}(Comparison Principle)
If $u,v\in\mathcal{C}^2(\Omega)\cap\mathcal{C}(\overline{\Omega})$ satisfy $u,v>0$ in $\overline{\Omega}$ and
\begin{equation}
\forall(x,y)\in\Omega\qquad \begin{array}{c}
                             u(a,y)<u(x,y)<u(b,y)\\
                             v(a,y)<v(x,y)<v(b,y)
                            \end{array}
\label{eq:NLcompcondition}
\end{equation}
then
$$
\left.
\begin{array}{cccc}
\Phi(u) & \geq & \Phi(v) & (\Omega)\\
u & \geq & v & (\partial \Omega)\\
\displaystyle{\min_{x=b}}\;u& > & \displaystyle{\max_{x=a}}\;v  & 
\end{array}
\right\}
\quad\Rightarrow \quad u\geq v \quad \left(\overline{\Omega}\right).
$$
\label{theo:NLcomp}
\end{theo}
The proof is easily adapted from \cite{BeresNiren-sliding}: the fact that the equation is invariant under $x$-translations and that the domain is convex in this direction allows to compare translates of $u$ and $v$.
\par
Condition \eqref{eq:NLcompcondition} may seem quite restrictive at first glance, as it requires $u,v$ to lie strictly between their boundary values: the following proposition ensures that this holds for any positive classical solution of problem \eqref{eq:DLproblem}.
\begin{prop}
Any positive solution $p\in\mathcal{C}^2(D_L)\cap \mathcal{C}(\overline{D_L})$ of problem \eqref{eq:DLproblem} satisfies
$$
\forall(x,y)\in D_L\qquad p(-L,y)<p(x,y)<p(L,y).
$$
\label{prop:A<p<B}
\end{prop}
\begin{proof}
Assume that $p$ is such a solution: since $p>0$ on the (compact) cylinder $[-L,L]\times\T$, equation $-mp\D p+(c+\a)p_x-|\nabla p|^2=0$ can be considered as a uniformly elliptic equation $Lp=0$ with no zero-th order term: the classical weak Maximum Principle therefore implies on $\overline{D_L}$ that $A=\displaystyle{\min_{\partial D_L}}\;p\leq p \leq \displaystyle{\max_{\partial D_L}}\;p=B$, and the classical strong Maximum Principle ensures that the inequalities are strict in $D_L$.
\end{proof}

\begin{cor}
There exists at most one positive solution $p\in\mathcal{C}^2(D_L)\cap\mathcal{C}(\overline{D_L})$ of problem \eqref{eq:DLproblem}.
\label{cor:uniqueness}
\end{cor}
\begin{proof}
Assume $p_1\neq p_2$ are two different solutions: $\Phi(p_1)=\Phi(p_2)=0$, $\displaystyle{\min_{x=b}}\;p_i=B>A=\displaystyle{\max_{x=a}}\;p_j$ and by previous proposition $p_1,p_2$ satisfy condition \eqref{eq:NLcompcondition}: Theorem \ref{theo:NLcomp} yields $p_i\geq p_j$ and therefore $p_1=p_2$.
\end{proof}
We recall that a function $p^+\in\mathcal{C}^2(\Omega)\cap\mathcal{C}(\overline{\Omega})$ (resp. $p^-$) is a supersolution (resp. subsolution) if $\Phi(p^+)\geq 0$ (resp. $\Phi(p^-)\leq 0$), and construct below two different types of planar sub and supersolutions.
\par
An elementary computation shows that a planar affine function  $p^+(x,y)=A^+x+B^+$ is a supersolution (resp. $p^-(x,y)=A^-x+B^-$ is a subsolution) if and only if $0+(c+\a)A^+\geq(A^+)^2\qquad(\text{resp. }A^-,\leq)$. Due to hypothesis \eqref{hyp:c0<c+a<c1} this condition is satisfied as soon as $0\leq A^+\leq c_0$ (resp. $A^-\geq c_1$ or $A^-\leq 0$): any affine function with positive slope $A^+\leq c_0$ (resp. $A^-\geq c_1$) is hence a supersolution (resp. subsolution).
\par
We will also use some additional planar sub and supersolutions. For any $x_0\in\R$, $M>\delta>0$ and $C>0$ consider the following boundary value problem
$$
u_C(x):\qquad
\left\{
\begin{array}{rcc}
-muu''+Cu' & = & (u')^2,\\
 u(-\infty) & = & \delta,\\
 u(x_0) & = & M.
\end{array}
\right. 
$$
Using ODE techniques one easily shows that there exists a unique solution $u_C$, which satisfies $u_C>\delta$, $C>u_C'>0$ and $u_C''>0$ for $x\in\R$. Defining $p^+(x,y):=u_C(x)$ one easily computes for $0<C\leq c_0$ (hence $c+\alpha(y)\geq C$)
$$
 \Phi(p^+) \geq  -m u_C u_C''+ C u_C' -(u_C')^2  = 0,
$$
and $p^+$ is therefore a planar supersolution. The same computation shows that if $c+\alpha(y)\leq c_1\leq C$ then $p^-(x,y)=u_C(x)$ is a planar subsolution $\Phi(p^-)\leq 0$.
\par
This allows us to build planar sub and supersolutions tailored to \eqref{eq:DLproblem} as follows. Let $\delta>0$ be a small elliptic regularization parameter, and define
\begin{equation}
p^+(x,y):=u_{c_0}(x),\qquad\left\{
\begin{array}{rcl}
-muu''+c_0u' & = & (u')^2,\\
 u(-\infty) & = & \delta,\\
 u(0) & = & 1.
\end{array}
\right.
\label{eq:defp+}
\end{equation}
If
\begin{equation}
B:=p^+(L),
\label{eq:defBrightboundary}
\end{equation}
similarly define
\begin{equation}
p^-(x):=u_{c_1}(x),\qquad\left\{
\begin{array}{rcl}
-muu''+c_1u' & = & (u')^2,\\
 u(-\infty) & = & \delta,\\
 u(L) & = & B.
\end{array}
\right. 
\label{eq:defp-}
\end{equation}
As discussed above $p^-, p^+$ are a planar sub and supersolution on $D_L=[-L,+L]\times\T$.  They satisfy all the hypotheses of our Comparison Theorem~\ref{theo:NLcomp}, and therefore $p^-\leq p^+$.

We prove below that, choosing
\begin{equation}
A:=\cfrac{p^+(-L)+p^-(-L)}{2},
\label{eq:defAleftboundary}
\end{equation}
there exists at least one solution $p$ of problem \eqref{eq:DLproblem} lying between $p^-$ and $p^+$. 
\begin{theo}
Fix $\delta>0$ small enough: for $L>0$ large enough and $A,B$ defined by \eqref{eq:defAleftboundary}-\eqref{eq:defBrightboundary} there exists a unique positive classical solution $p\in\mathcal{C}^{2}\left(D_L\right)\cap\mathcal{C}^1\left(\overline{D_L}\right)$ of \eqref{eq:DLproblem}. Moreover, it satisfies $p^-(x)\leq p(x,y)\leq p^+(x)$ on $\overline{D_L}$ and $p\in\mathcal{C}^{\infty}\left(D_L\right)$.
\label{theo:existsp-leqpleqp+}
\end{theo}
\begin{proof}
Uniqueness is given by corollary \ref{cor:uniqueness}. It was shown in~\cite{ChoquetLeray-quasilin} that if there exist {\it strict} sub and supersolutions $p^-<p^+$ then there is a classical solution $p$ satisfying $p^-\leq p\leq p^+$. Note, however, that we set $C=c_0,c_1$ in \eqref{eq:defp+}-\eqref{eq:defp-}. This corresponds to non-strict inequalities $\Phi(p^+)\geq 0$ and $\Phi(p^-)\leq0$, meaning that these particular sub and super solutions are not strict ones. It is however easy to approximate $p^{\pm}$ by \textit{strict} sub and supersolutions $p^{\pm}_{\eps}$, where $p^+_{\eps}>p^+$ and $p^-_{\eps}<p^-$ are such that $p^{\pm}_{\eps}\rightarrow p^{\pm}$ uniformly on $D_L$ when $\eps\rightarrow 0^+$; this can be done setting $c_0-\eps$ instead of $c_0$ and $c_1+\eps$ instead of $c_1$ in \eqref{eq:defp+}-\eqref{eq:defp-} (and also approximating boundary conditions). 

All the hypotheses of Theorem 1 in \cite{ChoquetLeray-quasilin} are here easily checked, and we conclude that there exists at least one solution $p_{\eps}\in\mathcal{C}^{2,\a}(D_L)\cap\mathcal{C}^1(\overline{D_L})$ such that $p_{\eps}^-\leq p_{\eps}\leq p_{\eps}^+$ on $\overline{D_L}$ and satisfying boundary conditions $p_{\eps}(-L,y)=A$, $p_{\eps}(+L,y)=B$. By uniqueness (corollary \ref{cor:uniqueness}) this solution is independent of $\eps$, i-e $p_{\eps}=p$; taking the limit $\eps\rightarrow 0^+$ yields $ p^-\leq p\leq p^+$ on $\overline{D_L}$, and $p$ is smooth on $D_L$ by standard elliptic regularity.
\end{proof}
As we let $L \to \infty$ in the next section, we need monotonicity of $p$ in the $x$ direction, as well as an estimate on $p_x$ uniformly in $L$, the size of the cylinder $D_L$.
\begin{prop}
The solution $p(x,y)$ of \eqref{eq:DLproblem} satisfies
\begin{equation}
 0< p_x\leq c_1
\label{eq:0leqpxleqc1}
\end{equation}
on $\overline{D_L}$, where $c_1>0$ is given in~\eqref{hyp:c0<c+a<c1}.
\label{prop:0<px<c1finitedomain}
\end{prop}
\begin{proof}
$p\in\mathcal{C}^{\infty}(D_L)\cap\mathcal{C}^1\left(\overline{D_L}\right)$ is smooth enough to differentiate \eqref{eq:PME} with respect to $x$, and $q=:p_x\in\mathcal{C}^{\infty}(D_L)\cap\mathcal{C}\left(\overline{D_L}\right)$ satisfies
\begin{equation}
-mp\D q+[(c+\a)q_x-2\nabla p\cdot\nabla q]-(m\D p)q=0.
\label{eq:eqpx}
\end{equation}
\par
We first prove the upper estimate $p_x\leq c_1$. Since we set $p(-L,y)=A<B=p(L,y)$ there exists at least a point inside $D_L$ where $p_x>0$; any potential maximum interior point for $q=p_x$ therefore satisfies $q>0$, and of course $\nabla q=0$, $\D q\leq 0$. Using \eqref{eq:eqpx} we compute at such a positive interior maximum point $(m\D p)q=-mp\D q\geq 0$, hence $m\D p\geq 0$ and $mp\D p\geq 0$. The original equation \eqref{eq:PME} satisfied by $p>0$ yields now
$$
0\leq mp\D p=(c+\a)p_x-|\nabla p|^2\quad\Rightarrow\quad (p_x)^2\leq |\nabla p|^2\leq (c+\a)p_x,
$$
and since $q=p_x>0$ at this maximum point we obtain $q=p_x\leq (c+\alpha)\leq c_1$.
\par
We just controlled any potential maximum value for $p_x$ inside the cylinder, and we control next $p_x$ on the boundaries using sub and supersolutions as barriers for $p$. Recall that the boundary values are flat, $p(-L,y)=A$ and $p(L,y)=B$.
\par
On the right side we use the previous subsolution $p^{-}(x)$ as a barrier from below: recalling that $p^-_x\leq c_1$ and that $p^-$ and $p$ agree at $x=L$ we obtain $p_x(L,y)\leq p^-_x(L)\leq c_1$.

On the left we use a new planar supersolution as a barrier from above: let $\overline{p}(x)$ be the unique affine function connecting $\overline{p}(-L)=A$ and $\overline{p}(L)=B$, with slope $s=\frac{B-A}{2L}$. Using \eqref{eq:defp+}-\eqref{eq:defp-} it is easy to see that $B\sim c_0L$ and $A\sim\delta$ when $L\to +\infty$ for fixed $\delta$. As a consequence $s\sim c_0/2<c_0$ for $L$ large enough, and $\overline{p}$ is indeed a supersolution. Since $p$ agrees with $\overline{p}$ on both boundaries our comparison Theorem \ref{theo:NLcomp} ensures that $p\leq\overline{p}$ on $D_L$, thus $p_x\leq s\leq c_0\leq c_1$ on the left boundary.

The lower estimate $q=p_x>0$ is inside $D_L$ a classical consequence of the Sliding Method \cite{BeresNiren-sliding}. In order to establish this strict monotonicity up to the boundaries we consider $-mp\D p+(c+\a)p_x-|\nabla p|^2=0$ as a linear elliptic equation $Lp=0$ with trivial zero-th order coefficient. Proposition \ref{prop:A<p<B} and flatness of the boundaries show in addition that $p$ attains a strict minimum at any point of the left boundary, and also a strict maximum at any point of the right boundary. Hopf Lemma then implies that $p_x>0$ on both boundaries.

\end{proof}
\begin{prop} (Uniform Pinning) There exists large constants $K>0$, $K_1\approx K-\sqrt{K}$ and $K_2\approx K+\sqrt{K}$ such that, for any $L$ large and any $\delta$ small enough, there exists $x^*=x^*(L,\delta)\in]0,L[$ such that
\begin{enumerate}
\item
$\displaystyle{\lim_{L\rightarrow +\infty}} \left( L-x^* \right)= +\infty$,
\label{item:X-L->infty}
\item
$K_1\leq p(x^*,y)\leq K_2$.
\end{enumerate}
The constants $K,K_1,K_2$ depend on the upper bound $c_1$ in~\eqref{hyp:c0<c+a<c1} and $m>0$, but not on $L$ or $\delta$.
\label{prop:pinning}
\end{prop}
	
\begin{proof}
The idea is as follows. When $x$ increases from $-L$ to $L$ the map $x\mapsto\int\limits_{\T}p(x,y)\mathrm{d}y$ increases from $A\sim \delta\leq 1$ to $B\sim c_0 L\gg 1$. For fixed large $K$ and $L$ large enough this integral has to take on the value $K$ at least for some $x\in]-L,L[$. The equation for $p$ then allows us to control the $y$-oscillations of $p$ along this line by $\mathcal{O}(\sqrt{K})$. If $K$ is chosen large enough these oscillations are small compared to the average, and $p$ along this line will therefore be of the same order $\mathcal{O}(K)$ than its average $\int p\mathrm{d}y$. This will be our pinning line $x=x^*$ (up to some further small translation).
\begin{itemize}
 \item 
Choose a large constant $K>1$, and for $x\in[-L,L]$ define $F(x):=\int\limits_{\T}p(x,y)\mathrm{d}y$: since $p(0,y)\leq p^+(0)=1$ we have that $F(0)<K$. By convexity $p^-$ lies above its tangent plane $t_L(x)$ at $x=L$, and we recall that we had set $p^-(L)=p(L,y)=p^+(L)=B$. For $L$ large and $\delta$ small $t_L(x)=K$ has a unique solution $x=x_K$ given by $x_K=L+\frac{K-B}{p^-_x(L)}$, and $F(x_K)\geq p^-(x_K)\geq t_L(x_K)=K$. Remarking that $F$ is increasing ($p_x>0$), there exists a unique $x^*_K(L,\delta)\in]0,x_K]$ such that
$$
F(x^*_K)=\int\limits_{\T} p(x^*_K,y)\mathrm{d}y=K.
$$
Manipulating \eqref{eq:defp+}-\eqref{eq:defp-} it is easy to check that for $K,\delta$ fixed and $L\rightarrow +\infty$ there holds
$$
\left.
\begin{array}{ccc}
B=p^+(L) & \sim & c_0L\\
p^-_x(L) & \sim & c_1
\end{array}
\right\}
\Rightarrow x_K=L+\frac{K-B}{p^-_x(L)}\sim \left(1-\frac{c_0}{c_1}\right)L;
$$
as a consequence of \eqref{hyp:c0<c+a<c1} the line $x=x^*_K(\delta,L)$ stays away from both boundaries.
\item
Let us now slide the whole picture to the left by setting $\tilde{p}(x,y)= p(x+x^*_K,y)$, so that $x=x^*_K$ corresponds in this new frame to $x=0$; for simplicity of notation we will use  $p(x,y)$ instead of $\tilde{p}(x,y)$ below. The corresponding domain still grows in both directions when $L\rightarrow+\infty$, and by definition of $x^*_K$ we have that
$$
\displaystyle{\int\limits_{\T} p(0,y)\mathrm{d}y}=K.
$$
We claim now that there exists a constant $C$, depending only on $m\neq 1$ and the upper bound for the flow $c_1$, such that
\begin{equation}
\forall x>0,\qquad \displaystyle{\iint\limits_{[0,x]\times\T}|\nabla p|^2\mathrm{d}x\mathrm{d}y}\leq C(K+x).
\label{eq:gradpccontrol}
\end{equation}
Indeed, integrating by parts the Laplacian term in $-mp\D p+(c+\a)p_x=|\nabla p|^2$ over a subdomain $\Omega=[0,x]\times \T$ and combining the resulting $|\nabla p|^2$ term with the one on the right hand side yields
\begin{equation}	
(m-1)\displaystyle{\iint\limits_{\Omega}|\nabla p|^2\mathrm{d}x\mathrm{d}y}+m\displaystyle{\int\limits_{\T} pp_x(0,y)\mathrm{d}y}-m\displaystyle{\int\limits_{\T} pp_x(x,y)\mathrm{d}y}+\displaystyle{\iint\limits_{\Omega}(c+\a)p_x\mathrm{d}x\mathrm{d}y}=0.
\label{eq:IBP_pinning}
\end{equation}
\begin{enumerate}
 \item
If $m-1>0$ we use $m\int\limits_{\T} pp_x(0,y)\mathrm{d}y\geq 0$ and $\iint\limits_{\Omega}(c+\a)p_x\mathrm{d}x\mathrm{d}y\geq 0$ in \eqref{eq:IBP_pinning}. This leads to $(m-1)\iint\limits_{\Omega}|\nabla p|^2\mathrm{d}x\mathrm{d}y\leq m\int\limits_{\T} pp_x(x,y)\mathrm{d}y$, and since $0<p_x\leq c_1$
$$
\displaystyle{\iint\limits_{\Omega}|\nabla p|^2\mathrm{d}x\mathrm{d}y}\leq \frac{mc_1}{m-1}\displaystyle{\int\limits_{\T} p(x,y)\mathrm{d}y}.
$$
Our monotonicity estimate $0<p_x\leq c_1$ again yields
$$
\displaystyle{\int\limits_{\T} p(x,y)\mathrm{d}y}=\displaystyle{\int\limits_{\T} p(0,y)\mathrm{d}y}+\displaystyle{\iint\limits_{\Omega}p_x\mathrm{d}x\mathrm{d}y}\leq K+c_1 x,
$$
and together with the previous inequality
$$
\forall x>0,\qquad \displaystyle{\iint\limits_{[0,x]\times\T}|\nabla p|^2\mathrm{d}x\mathrm{d}y}\leq\frac{mc_1}{m-1}(K+c_1 x)\leq C(K+x).
$$
\item 
If $0<m<1$ we use $pp_x(x,y)>0$ in \eqref{eq:IBP_pinning} to obtain
$$
(1-m)\displaystyle{\iint\limits_{\Omega}|\nabla p|^2\mathrm{d}x\mathrm{d}y}\leq m\displaystyle{\int\limits_{\T} pp_x(0,y)\mathrm{d}y}+\displaystyle{\iint\limits_{\Omega}(c+\a)p_x\mathrm{d}x\mathrm{d}y}.
$$
Since $0<p_x\leq c_1$ and $0<c+\alpha\leq c_1$ this leads to
$$
\begin{array}{rcl}
\displaystyle{\iint\limits_{\Omega}|\nabla p|^2\mathrm{d}x\mathrm{d}y} & \leq &  \frac{mc_1}{1-m}\displaystyle{\int\limits_{\T} p(0,y)\mathrm{d}y}+\frac{1}{1-m}\displaystyle{\iint\limits_{\Omega}c_1^2\mathrm{d}x\mathrm{d}y}\\
 & \leq & C(K+x)
\end{array}
$$
with $C=\frac{1}{1-m}\max (mc_1,c_1^2)$ depending only on $m$ and $c_1$.
\end{enumerate}
\item
In the spirit of \cite{ConstantinKiselevRyzhik-quenching} we control now the oscillations $O(x)=\displaystyle{\left|\max_{y\in\T} p(x,y)-\min_{y\in\T} p(x,y)\right|}$ in the $y$ direction: by Cauchy-Schwarz inequality we have that
\begin{eqnarray*}
O^2(x)   \leq \left( \int\limits_{\T} |p_y(x,y)|\mathrm{d}y\right)^2 \leq  \int\limits_{\T} |p_y(x,y)|^2\mathrm{d}y  \leq  \int\limits_{\T} |\nabla p|^2(x,y)\mathrm{d}y,
\end{eqnarray*}
and integrating from $x=0$ to $x=1$ with \eqref{eq:gradpccontrol} leads to
$$
\begin{array}{rclcl}
(1-0)\displaystyle{\min_{x\in[0,1]}O^2(x)} & \leq & \displaystyle{\int_0^1 O^2(x)\mathrm{d}x} & &\\
 & \leq &  \displaystyle {\iint\limits_{[0,1]\times\T}|\nabla p|^2\mathrm{d}x\mathrm{d}y} & \leq & C(K+1).
\end{array}
$$
Let now $x^*\in[0,1]$ be any point where $O^2(x)$ attains its minimum on this interval; along the particular line $x=x^*$ the last inequality yields
\begin{equation}
O(x^*)\leq \sqrt{C(K+1)}
\label{eq:O(X)leqsqrt(K)}
\end{equation}
and these oscillations are therefore controlled uniformly in $L$ ($C$ depends only on $m$ and $c_1$). Moreover, $x^*\in[0,1]$ and $0<p_x\leq c_1$ control $p$ in average from below and from above
\begin{equation}
K=\displaystyle{\int\limits_{\T}p(0,y)\mathrm{d}y}\leq\displaystyle{\int\limits_{\T}p(x^*,y)\mathrm{d}y}\leq K+c_1 x^*\leq K+c_1.
\label{eq:intp=K}
\end{equation}
\item
For $K$ large enough but fixed \eqref{eq:O(X)leqsqrt(K)},~\eqref{eq:intp=K}  imply $0<K_1\leq p(x^*,y)\leq K_2$ as desired, with $K_1\approx K-\mathcal{O}(\sqrt{K})$ and $K_2\approx K+\mathcal{O}(\sqrt{K})$ up to constants depending only on $c_1$ and $m$. Finally $x^*\in[0,1]$ may depend on $L,\delta,c_1$ (and actually does) but stays far enough from both boundaries, so that the new translated domain still grows to infinity in both directions when $L\rightarrow +\infty$.
\end{itemize}
\end{proof}
%
%
\section{$\delta$-solutions on the infinite cylinder}
\label{section:unif_infinite_domain}
From now on we will work in the translated frame $D_L=]-L-x^*,L-x^*[\times \T$,
where $x^*=x^*(L,\delta)$ is defined as in proposition \ref{prop:pinning} above.
Since the domain depends on $L$, the solution depends on $L$ as well. We emphasize that by writing $p=p^L$ ($\delta>0$ is fixed so we may just omit the dependence on $\delta$), 
and let also set $D=\R\times \T$ to be the infinite cylinder.
\begin{theo}
Up to a subsequence we have $p^L\rightarrow p$ in $\mathcal{C}^{2}_{loc}(D)$ when $L\rightarrow +\infty$, where $p\in \mathcal{C}^{\infty}\left(D\right)$ is a classical solution of $-mp\D p+(c+\a)p_x=|\nabla p|^2$. This limit $p$ satisfies
\begin{enumerate}
 \item $0\leq p_x\leq c_1$
\item $p\geq \delta$
\item $p$ is nontrivial: $K_1\leq p(0,y)\leq K_2$
\end{enumerate}
where $K_1,K_2$ are the pinning constants in proposition \ref{prop:pinning}.
\label{theo:pL->p}
\end{theo}
\begin{proof}
Using interior $L^q$ elliptic regularity arguments for fixed $q>d=2$ we will obtain $W^{3,q}$ estimates on $p^L$, and this will allow us to retrieve the strong convergence $p^L\rightarrow p$ in $\mathcal{C}^2_{loc}$.
 \par
The most difficult term to estimate is $|\nabla p|^2$. We handle it using a different unknown which appears very naturally in the original setting~\eqref{eq:PME0}, namely
\begin{equation}
w:=\frac{m^2}{m+1}p^{\frac{m+1}{m}}=m\left(\cfrac{m+1}{m}\right)^{\frac{1}{m}}T^{m+1}
\label{eq:algebraicpw}
\end{equation}
An easy computation shows that this new unknown satisfies on $D_L$ a classical Poisson equation
\begin{equation}
\D w^L=f^L,
\label{eq:poissonwL}
\end{equation}
where the non-homogeneous part
\begin{equation}
f^L:=(c+\a)\left(p^L\right)^{\frac{1}{m}-1}p^L_x
\label{eq:deffL}
\end{equation}
involves only $p^L$ and $p^L_x$, on which we have local $L^{\infty}$ control uniformly in $L$. Indeed, $p^L$ is pinned at $x=0$ by $K_1\leq p^L(0,y)\leq K_2$ and cannot grow too fast in the $x$ direction because of $0\leq p^L_x\leq c_1$.
\par
If $m<1$ the exponent $\frac{1}{m}-1$ in \eqref{eq:deffL} is positive and we control $f^L$ uniformly in $L$ on any compact set. However, if $m>1$, this exponent is negative and we need to bound $p_L$ away from zero uniformly in $L$. For $\delta> 0$ fixed this is easy since  we constructed $p^L \geq p^->\delta>0$, but this will be a problem later when taking the limit $\delta\rightarrow 0$ (see next section).
\par
As a consequence, for any fixed $q>d=2$, $f^L$ is in $L^q$ on any bounded subset $\Omega\subset D$ and we control
$$
||f^L||_{L^q(\Omega)}\leq C
$$
uniformly in $L$ ($C$ may of course depend on $\Omega$, $q$ and $\delta$). Since $w^L$ is defined as a positive power of $p^L$ and $p^L$ is locally controlled in the $L^{\infty}$ norm uniformly in $L$ the same holds for $w^L$,
$$
||w^L||_{L^q(\Omega)}\leq C.
$$
Let $\Omega=]-a,a[\times\T\subset D$ and $K=\overline{\Omega}$; let also $\Omega_2=]-2a,2a[\times\T$ and $\Omega_3=]-3a,3a[\times \T$ so that $\Omega\subset\subset\Omega_2\subset\subset\Omega_3$. By interior $L^q$ elliptic regularity for strong solutions (the version we use here is \cite{GilbargTrudinger}, Theorem 9.11 p.235) there exists a constant $C$ depending only on $\Omega_2$, $\Omega_3$ and $q$ such that
$$
||w^L||_{W^{2,q}(\Omega_2)}\leq C \left(||w^L||_{L^q(\Omega_3)}+||f^L||_{L^q(\Omega_3)}\right).
$$
As discussed above we control $w^L$ and $f^L$, hence
\begin{equation}
||w^L||_{W^{2,q}(\Omega_2)}\leq C
\label{eq:wLW2,qleqC}
\end{equation}
for some $C>0$ depending only on $\Omega_3,\Omega_2$ and $q$.

The next step is using \eqref{eq:algebraicpw}-\eqref{eq:deffL} to express $f^L$ only in terms of $w^L$
$$
f^L=\cfrac{c+\a}{m+1}(w^L)^{-\frac{m}{m+1}}w^L_x.
$$
Expressing $\nabla f^L$ only in terms of $w^L$, $\nabla w^L$ and $D^2w^L$ (which are controlled by $||w^L||_{W^{2,q}(\Omega_2)}$), using the lower bound $p^L\geq\delta>0$ and uniform control on $p^L$, \eqref{eq:wLW2,qleqC}
implies that
$$
||\nabla f^L||_{L^q(\Omega_2)}\leq C
$$
for some $C$ depending only on the size $a$ of $\Omega$. Differentiating~\eqref{eq:poissonwL} implies
$$
\Delta(\partial_i w^L)=\partial_if^L,\qquad i=1,2.
$$
Repeating the previous $L^q$ interior regularity argument on $\Omega\subset\subset \Omega_2$ yields
$$
||\partial_i w^L||_{W^{2,q}(\Omega)}\leq C\left(||\partial_iw^L||_{L^q(\Omega_2)}+||\partial_if^L||_{L^q(\Omega_2)}\right)\leq C,
$$
and our previous estimate for $\nabla f^L$ together with \eqref{eq:wLW2,qleqC} finally yield the higher estimate
$$
||w^L||_{W^{3,q}(\Omega)}\leq C.
$$
The set $K=\overline{\Omega}=[-a,a]\times\T$ is bounded and the exponent $q$ was chosen larger than the dimension $d=2$. Thus compactness of the Sobolev embedding $W^{3,q}(\Omega)\subset\subset\mathcal{C}^2(K)$ implies, up to a subsequence, that
$$
w^L\overset{\mathcal{C}^2(K)}{\longrightarrow} w
$$
when $L\rightarrow +\infty$. By the diagonal extraction of a subsequence we can assume that the limit $w$ does not depend on the compact $K$. It means $w^L\rightarrow w$  in
 $\mathcal{C}^2_{loc}$ on the infinite cylinder $D$. The algebraic relation \eqref{eq:algebraicpw} and $p^L\geq \delta>0$ imply that
$$
p^L\overset{\mathcal{C}^2_{loc}(D)}{\longrightarrow} p.
$$
It implies that we can take the pointwise limit in the nonlinear equation.  The limit $p$ solves therefore the same equation $-mp\D p+(c+\a)p_x=|\nabla p|^2$ on the infinite cylinder.

The remaining estimates are easily obtained by taking the limit in $0\leq p^L_x\leq c_1$, $\delta<p^-\leq p^L$ and in the pinning proposition \ref{prop:pinning}.
Lastly, $p$ is smooth by classical elliptic regularity.
\end{proof}
\begin{prop}
We have $\displaystyle{\lim_{x\rightarrow -\infty}}p(x,y)=\delta$ uniformly in $y$.
\label{prop:pdelta(-infty)=delta}
\end{prop}
\begin{proof}
The previous lower barrier $\delta< p^L$ on $D_L$ immediately passes to the limit $L\rightarrow +\infty$, and
\begin{equation}
\forall(x,y)\in D,\qquad p\geq \delta.
\label{eq:pgeqdelta}
\end{equation}
In order to estimate $p$ from above let us go back to the untranslated frame $x\in[-L,L]$ and remark that by definition $p^+$ does not depend on $L$, see \eqref{eq:defp+}. An easy computation shows that $p^+(-L)\rightarrow \delta$ when $L\rightarrow +\infty$. The subsolution $p^-$ actually depends on $L$ through boundary condition, see \eqref{eq:defp-}, but using the monotonicity $\partial_xp^->0$ is is quite easy to prove that $p^-(-L)\sim \delta$ when $L\rightarrow +\infty$. The left boundary condition consequently reads
$$
p^L(-L,y)=A=\frac{p^+(-L)+p^-(-L)}{2}\underset{L\rightarrow+\infty}{\sim}\delta.
$$
However, the limit $\displaystyle{\lim_{x\rightarrow -\infty}}p(x,y)\overset{??}{=}\displaystyle{\lim_{L\rightarrow +\infty}}p^L(-L,y)$ is not clear because the convergence $p^L\rightarrow p$ is only local on compact sets (and also because we translated from one frame to another).
\par
In order to circumvent this technical difficulty we move back to the translated frame and build for $x\in]-L-x^*,0[\times\T$ a family of planar supersolutions $\overline{p}_{\eps}(x)$ independent of $L$ such that $\overline{p}_{\eps}(-\infty)=\delta+\eps$. The construction is the following: fix $\eps>0$ and define $\overline{p}_{\eps}(x)$ as the unique solution of Cauchy problem
\begin{equation}
\overline{p}_{\eps}(x):\qquad
\left\{
\begin{array}{rcl}
 -muu''+c_0 u' & = & (u')^2\\
 u(0) & = & 2K_2\\
 u(-\infty) & = & \delta+\eps
\end{array}
\right.,
\label{eq:defpeps+}
\end{equation}
where $K_2$ is the constant in proposition \ref{prop:pinning} such that $p^L(0,y)\leq K_2$. As already computed the setting $C=c_0\leq c+\a$ in \eqref{eq:defpeps+} implies that $\overline{p}_{\eps}$ is a supersolution. By monotonicity both $p^L$ and $\overline{p}_{\eps}$ satisfy condition \eqref{eq:NLcompcondition}, and for $L$ large and $\delta,\eps$ small it is easy to check that $p\leq \overline{p}_{\eps}$ on the boundaries $x=-L-x^*,0$: Theorem \ref{theo:NLcomp} guarantees that
$$
\forall(x,y)\in]-L-x^*,0[\times\T,\qquad p^L \leq \overline{p}_{\eps}.
$$
For $\delta,\eps$ fixed, $\overline{p}_{\eps}$ is independent of $L$: taking the limit $L\rightarrow +\infty$ yields
\begin{equation}
\forall(x,y)\in]-\infty,0[\times\T,\qquad p(x,y)\leq \overline{p}_{\eps}(x).
\label{eq:pleqpbar}
\end{equation}
Taking now the limit $\eps\rightarrow 0$ in \eqref{eq:defpeps+}, it is easy to show that $\overline{p}_{\eps}(x)\rightarrow \overline{p}(x)$ uniformly on $]-\infty,0]$, where $\overline{p}$ is the solution of the same Cauchy Problem as $\overline{p}_{\eps}$ - except for $\overline{p}(-\infty)=\delta$ instead of $\overline{p}_{\eps}(-\infty)=\delta+\eps$ - and satisfies $\displaystyle{\lim_{x\rightarrow -\infty}}\overline{p}(x)=\delta$. Combining the limit $\eps\rightarrow 0$ in \eqref{eq:pleqpbar} with the lower barrier \eqref{eq:pgeqdelta} we finally obtain
$$
\forall (x,y)\in]-\infty,0]\times\T,\qquad \delta\leq p(x,y)\leq \underbrace{\overline{p}(x)}_{\rightarrow \delta}
$$
as desired.
\end{proof}
\begin{rmk}
The proof above actually implies a stronger statement than $\displaystyle{\lim_{x\rightarrow -\infty}} p(x,y)=\delta$, namely $\delta\leq p\leq \overline{p}$ for $x\rightarrow -\infty$: just working on the ODE $-m\overline{p}\overline{p}''+c_0\overline{p}'=\left(\overline{p}'\right)^2$ it is straightforward to obtain the exponential decay $|p-\delta|=\mathcal{O}\left(e^{c_0x/m\delta}\right)$. The exponential rate $c_0/m\delta$ degenerates when $\delta\rightarrow 0^+$, which is consistent with the fact that a free boundary appears in this limit (see next section).
\end{rmk}

As stated in Theorem~\ref{theo:pL->p} the limit $p$ is non-decreasing in the $x$ direction (as a limit of increasing functions $p^L$). We establish below the strict monotonicity.
\begin{prop}
$p_x>0$ on the infinite cylinder.
\label{prop:p_x>0}
\end{prop}
\begin{proof}
The argument is very similar to the proof of Proposition \ref{prop:0<px<c1finitedomain}: differentiating the equation for $p$ with respect to $x$ yields a linear uniformly elliptic equation satisfied by $q=p_x\geq 0$. The classical Minimum Principle implies that either $q>0$ everywhere either $q\equiv 0$, and latter would contradict $p(-\infty,y)=\delta <K_2\leq p(0,y)$.
\end{proof}

%
%
\section{Limit $\delta\rightarrow 0$ and the free boundary}
\label{section:interface}
In the previous section we constructed for any small $\delta>0$ a nontrivial solution $p=\underset{L\rightarrow +\infty}{\lim}p^L$ of $-mp\D p+(c+\a)p_x=|\nabla p|^2$ on the infinite cylinder $D=\R\times\T$, satisfying the uniform ellipticity condition $p>\delta>0$. Let us now write $p=p^{\delta}$ in order to stress the dependence on $\delta$. The next step is now to take the limit $\delta\rightarrow 0^+$ ($\delta$ is an elliptic regularization parameter), yielding the desired viscosity solution.
\par
Viscosity solutions are defined as follows: for $\delta>0$ let $E_{\delta}\subset\mathcal{C}^{0}(D)$ be the set of positive smooth solutions $p$ satisfying
\begin{enumerate}
 \item $\displaystyle{\lim_{x\rightarrow -\infty}} p(x,y)=\delta$ uniformly in $y$.
\item $p(x,y)\sim cx$ uniformly in $y$ at positive infinity.
\end{enumerate}
According to our Comparison Theorem \ref{theo:NLcomp} such solutions satisfy $p\geq \delta$, and let us recall that we refer to those as $\delta$-solutions or $\delta$-approximations. We define viscosity solutions as
\begin{defi}
 A function $p\in\mathcal{C}^0(D)$ is a viscosity solution if there exists a sequence $\left(p^{\delta}\right)_{\delta>0}\in E_{\delta}$ such that $\displaystyle{\lim_{\delta\rightarrow 0^+}}\;p^{\delta}=p$ in $\mathcal{C}^0_{loc}(D)$.
\label{defi:viscosity_solutions}
\end{defi}
\par
Let us comment on this definition, which is not the standard definition of viscosity solutions for second order equations: we will see below that this choice ensures uniqueness (see last section), but also regularity. Adapting the previous $L^q$ interior regularity argument, any viscosity solution $p$ will turn to be $\mathcal{C}^{\infty}$ on its positive set $D^+=\{p>0\}$ (see proof of Theorem \ref{theo:pdelta->p} below) which is not clear with the usual definition (in addition to being a difficult question, see e.g. \cite{CaffarelliCabre-FullyNL}). This definition can be seen as a particular case of the evanescent viscosity method, where we do not regularize the problem by modifying the equation itself as it is usually done. The regularization is actually performed through boundary conditions $p^{\delta}(-\infty,y)=\delta>0$ and $p^{\delta}(x,y)\underset{+\infty}{\sim} cx$, thus ensuring uniform ellipticity $p^{\delta}\geq\delta>0$. The delicate point is of course the loss of ellipticity when $\delta\rightarrow 0^+$.
\par
Our definition \ref{defi:viscosity_solutions} includes boundary conditions: any proper setting should consider the definition of the notion of solutions independently of boundary conditions. For the sake of simplicity we keep this definition, but let us point out that it can be relaxed into a proper definition.
\\
\par
Anticipating that $p=\displaystyle{\lim_{\delta\rightarrow 0}}\;p^{\delta}$ will have a free boundary, we cannot hope  convergence to hold in any $\mathcal{C}^k$ topology ($k\geq 1$) because of a potential gradient jump across the interface. In order to apply Arzel\`a-Ascoli Theorem we need bounds for $p^{\delta},\nabla p^{\delta}$ uniformly in $\delta$. At this stage we have pinned $0<K_1\leq p^{\delta}(0,y)\leq K_2$, and $0<p^{\delta}_x\leq c_1$ holds on the infinite cylinder: we therefore control $p^{\delta}$ and $p^{\delta}_x$ uniformly on any compact set, but we still have no control at all on $p^{\delta}_y$:
\begin{prop}
For any $a\geq 0$ there exists $C_a>0$ such that, for any small $\delta>0$,
$$
x\leq a\Rightarrow |p^{\delta}_y(x,y)|\leq C_a.
$$
\label{prop:estimatepy}
\end{prop}
\begin{proof}
We will first establish this estimate for $p^L$ on finite domains $[-L-x^*,a]\times\T$ by controlling $q=p^L_y$ at the boundaries and estimating the value of any potential interior extremal point. Taking the limit $L\rightarrow +\infty$ will then yield the desired estimate for $p^{\delta}=\underset{L\rightarrow +\infty}{\lim} p^L$. 
\begin{itemize}
\item
Fix $a\geq 0$: the uniform pinning $K_1\leq p^L(0,y)\leq K_2$ and monotonicity $0< p^L_x\leq c_1$ allow us to control $p^L$ uniformly in $\delta,L$ from above and away from zero on any small compact set $K=[a-\eps,a+\eps]\times\T$. Applying the previous $L^q$ interior elliptic regularity for $w=\frac{m^2}{m+1}p^{\frac{m+1}{m}}$ on the slightly larger set $\Omega_2=]a-2\eps,a+2\eps[\times\T\supset\supset\Omega:=\mathring{K}$ we obtain
$$
||w^L||_{W^{2,q}(\Omega)} \leq  C\left(||w^L||_{L^{q}(\Omega_2)}+||f^L||_{L^q(\Omega_2)}\right)\leq C_a\quad\Rightarrow\quad||p^L||_{\mathcal{C}^1(K)} \leq C_a
$$
for some constant $C_a$ depending only on $\Omega,\Omega_2$ and $q>2$ fixed, hence on $a$. It is here important that $p^L$ is bounded away from zero uniformly in $\delta$ on $\Omega_2$, see proof of Theorem \ref{theo:pL->p} for details. In particular
\begin{equation}
|p^L_y(a,y)|\leq C_a
\label{eq:pyleqCa}
\end{equation}
and the monotonicity estimate $0<p^L_x\leq c_1$ combined with the pinning yield
\begin{equation}
 x\leq a\quad\Rightarrow\quad 0< p^L(x,y)\leq C_a.
\label{eq:pleqCa}
\end{equation}
Differentiating \eqref{eq:PME} with respect to $y$ we see that $q^L:=p^L_y$ satisfies the linear elliptic equation
\begin{equation}
-mp^L\D q^L+\left[(c+\a)q_x^L-2\nabla p^L\cdot \nabla q^L\right]-\left(m\D p^L\right)q^L=-\a_yp^L_x.
\label{eq:PDEqL}
\end{equation}
Let $\Omega_a=]-L-x^*,a[\times \T$: on the left $x=-L-x^*$ we had a flat boundary condition $p^L(-x*-L,y)=cst$ so that $q^L(-L-x^*,y)=0$, and on the right boundary $x=a$ \eqref{eq:pyleqCa} holds. We therefore control $\left|q^L\right|=\left|p_y^L\right|\leq C_a$ on the boundaries.
\item
In order to control $p_y^L$ inside $\Omega_a$ we remark that any interior maximum point satisfies $q>0$ (unless by periodicity $p^L_y\equiv 0$, which is impossible if the flow $\a(y)$ is nontrivial), and of course $\D q^L\leq 0$, $\nabla q^L=0$. At such a maximum point \eqref{eq:PDEqL} immediately yields
$$
-\left(m\D p^L\right)q^L\leq -\a_yp^L_x;
$$
using $-mp^L\D p^L=|\nabla p^L|^2-(c+\a)p^L_x$ as well as the monotonicity estimate
$$
\begin{array}{ccccccc}
(q^L)^3-c_1^2 q^L & \leq &  \left[\left|\nabla p^L\right|^2-(c+\a)p^L_x\right]q^L &&&&\\
 & \leq & -\left(mp^L\D p^L\right)q^L & \leq & -\a_yp^Lp^L_x & \leq & C_a.
\end{array}
$$
Since at a maximum point $q^L>0$ this controls any potential maximum interior point $\displaystyle{\max_{(x,y)\in\Omega_a}}q^L(x,y)\leq C_a$ uniformly in $L,\delta$. A similar computation controls $q^L$ at any potential negative minimum point $\displaystyle{\min_{(x,y)\in\Omega_a}}q^L(x,y)\geq -C_a$, and combining with the previous boundary estimates yields
\begin{equation}
(x,y)\in[-L-x^*,a]\times\T\quad \Rightarrow \quad |p^L_y(x,y)|\leq C_a.
\label{eq:pLyleqCa}
\end{equation}
\end{itemize}
Theorem \ref{theo:pL->p} ensures that the convergence $p^L\rightarrow p^{\delta}$ holds in $\mathcal{C}^2_{loc}(D)$: taking the limit $L\rightarrow +\infty$ in \eqref{eq:pLyleqCa} finally yields the desired estimate for $p^{\delta}$.
\end{proof}
We can now state the main convergence result when $\delta\rightarrow 0^+$:
\begin{theo}
Up to a subsequence we have $p^{\delta}\rightarrow p$ in $\mathcal{C}^0_{loc}(D)$ when $\delta\rightarrow 0^+$, where $p\geq 0$ is continuous and nontrivial $\emptyset\neq D^+:=\{p>0\}$. Further:
\begin{enumerate}
\item $p$ is Lipschitz on any subdomain $]-\infty,a]\times \T$ (the Lipschitz constant may depend on $a$).
\item $p$ solves $-mp\D p+(c+\alpha)p_x=|\nabla p|^2$ in the viscosity sense on the infinite cylinder; $p|_{D^+}\in\mathcal{C}^{\infty}(D^+)$ is moreover a classical solution on $D^+$.
\item
$0<p_x\leq c_1$ on $D^+$.
\item $p$ has a free boundary $\Gamma:=\partial D^+\neq \emptyset$ and there exists an upper semi-continuous function $I(y)$ such that $p(x,y)>0 \Leftrightarrow x> I(y)$.
\item If $\underline{I}(y_0):=\displaystyle{\liminf_{y\rightarrow y_0}} I(y)<I(y_0)$, then at $y=y_0$ the free boundary is a vertical segment $\Gamma\cap\{y=y_0\}=[\underline{I}(y_0),I(y_0)]\times\{y=y_0\}$.
\end{enumerate}
\label{theo:pdelta->p}
\end{theo}
\begin{proof}
The pinning $K_1\leq p^{\delta}(0,y)\leq K_2$ and monotonicity $0<p^{\delta}_x\leq c_1$ control $p^{\delta}$ on any fixed compact set $K=[-a,a]\times\T$ uniformly in $\delta$. On this compact set $p^{\delta}_{y}$ is moreover bounded by proposition \ref{prop:estimatepy}: Arzel\`a-Ascoli Theorem guarantees that $p^{\delta}\rightarrow p$ uniformly on $K$ (up to extraction). Once again by diagonal extraction we can assume that the limit does not depend on the compact $K$, which means local uniform convergence
$$
p^{\delta}\overset{\mathcal{C}^0_{loc}(D)}{\rightarrow}p.
$$
$p$ is  nonnegative as a limit of positive functions, and non trivial since for example we had pinned $0<K_1\leq p^{\delta}(0,y)$.
\begin{enumerate}
\item
\label{item:CV_pdelta->p}
Proposition \ref{prop:estimatepy} and monotonicity $0<p^{\delta}_x\leq c_1$ yield $C_a$-Lipschitz estimates on $]-\infty,a]\times \T$ uniformly in $\delta$ for $p^{\delta}$: this passes to the $\mathcal{C}^0_{loc}$ limit $\delta\rightarrow 0^+$, and $p$ is therefore Lipschitz on any half cylinder $]-\infty,a]\times \T$.
\item
$p^{\delta}\in\mathcal{C}^2(D)$ was a classical solution of $-mp\D p+(c+\alpha)p_x=|\nabla p|^2$ on the infinite cylinder: according to our definition \ref{defi:viscosity_solutions} we need to check that $p^{\delta}$ grows as $cx$ when $s\rightarrow+\infty$. This is true but the proof is a long and technical computation: we will prove instead that the limit $p$ itself grows linearly (see section \ref{section:linear}). The proof of the linear growth is however exactly the same for $p^{\delta}$ and $p$, and we therefore admit here that the $\delta$-solutions grow linearly: $p$ is therefore a viscosity solution in the sense of Definition \ref{defi:viscosity_solutions}.
\begin{rmk}
Regardless of this linear growth issue, the limit $p$ is a viscosity solution in the classical sense as a consequence of usual stability theorems (see e.g. \cite{CrandallIshiiLions-userguide} \S 6). This is just the classical construction of evanescent viscosity solutions since we had uniform ellipticity $p^{\delta}\geq\delta>0$.
\end{rmk}
In order to prove the convergence $p^{\delta}\rightarrow p$ above we could not apply the same local $L^q$ interior elliptic regularity argument as in the proof of Theorem \ref{theo:pL->p}, mainly because we needed to bound $p^L$ away from zero (cf. the negative $p^L$ exponents $\frac{1}{m}-1$ for the non-homogeneous term in \eqref{eq:poissonwL}). This is of course impossible on the whole cylinder uniformly in $\delta$ because the equation degenerates when $\delta\rightarrow 0^+$ (this is indeed consistent with $p\equiv 0$ to the left of the free boundary, as claimed in our statement).

This strategy is however still efficient on the positive set $D^+=\{p>0\}$: indeed for any fixed compact subset $K\subset D^+$ we know a priori that the limit $p$ is positive, and therefore so is $p^{\delta}$ uniformly in $\delta\rightarrow 0^+$. This allows us to bound $p^\delta$ away from zero uniformly in $\delta$ on any compact set $K\subset D^+$ as
$$
\left.p^{\delta}\right|_K\geq C_K>0,
$$
where $C_K$ depends only on $K$. The interior $L^q$ regularity argument in the proof of Theorem \ref{theo:pL->p} then applies to the letter, and
$$
p^{\delta}\rightarrow\tilde{p}\quad\text{ in }\mathcal{C}^2_{loc}(D^+).
$$
The limit $\tilde{p}\in\mathcal{C}^2(D^+)$ is moreover a classical solution on $D^+$, and smooth by standard elliptic regularity. The previous convergence $p^{\delta}\overset{\mathcal{C}^0_{loc}(D)}{\rightarrow}p$ finally implies that $p|_{D^+}=\tilde{p}\in\mathcal{C}^{\infty}(D^+)$ is a classical solution on $D^+$.
\item
Convergence $p^{\delta}\rightarrow p$ is strong enough on $D^+$ to pass to the limit in $0<p^{\delta}_x\leq c_1$, so that $0\leq p_x\leq c_1$ on $D^+$. The strict monotonicity is obtained just as for the $\delta$-solutions: differentiating the equation for $p$ with respect to $x$ yields and elliptic equation $Lq=0$ satisfied by $q=p_x\geq 0$ on $D^+$ (where $p>0$ is smooth). Applying the Minimum Principle shows that either $q>0$, either $q\equiv 0$. Item \ref{item:exists_FB} below will show that $p$ actually vanishes far enough to the left: the pinning $0<K_1\leq p(0,y)$ then implies that $p$ has to increase at least somewhere in $D^+$, therefore excluding the case $q\equiv 0$.
\item
\label{item:exists_FB}
In order to show the existence of the free boundary $\Gamma=\partial\{p>0\}\neq \emptyset$ we build new suitable planar sub and supersolutions $p^{\delta,-}(x),p^{\delta,+}(x)$ for $p^{\delta}$ as follows: defining $p^{\delta,-},p^{\delta,+}$ to be the unique planar solutions of the following Cauchy problems
$$
p^{\delta,-}(x):\quad\left\{
\begin{array}{rcl}
-muu''+c_1u' & = & (u')^2\\
 u(-\infty) & = & \frac{\delta}{2}\\
 u(0) & = & K_1
\end{array}
\right.,
\hspace{1.5cm}
p^{\delta,+}(x):\quad\left\{
\begin{array}{rcl}
-muu''+c_0u' & = & (u')^2\\
 u(-\infty) & = & 2\delta\\
 u(0) & = & K_2
\end{array}
\right.,
$$
and we have of course $\Phi\left(p^{\delta,-}\right)\leq \Phi\left(p^{\delta}\right)=0\leq \Phi\left(p^{\delta,+}\right)$. Let us moreover recall from proposition \ref{prop:pdelta(-infty)=delta} that $\underset{x\rightarrow-\infty}{\lim}p^{\delta}(x,y)=\delta$ uniformly in $y$, so that $p^{\delta,-}< p^{\delta}< p^{\delta,+}$ when $x\rightarrow -\infty$. On the right boundary we set $p^{\delta,-}(0)=K_1\leq p^{\delta}(0,y)\leq K_2= p^{\delta,+}(0)$: applying Theorem \ref{theo:NLcomp} on $]-\infty,0]\times\T$ yields
\begin{equation}
x\leq 0\quad \Rightarrow \quad p^{\delta,-}(x)\leq p^{\delta}(x,y)\leq p^{\delta,+}(x)
\label{eq:pdelta-leqpdeltaleqpdelta+}
\end{equation}
(note that $p^{\delta,-}_x,p^{\delta,+}_x,p^{\delta}_x>0$ so that condition \ref{eq:NLcompcondition} does hold). When $\delta\rightarrow 0^+$ one can prove that
$$
p^{\delta,-}(x)\rightarrow p^-(x):=[K_1+c_1x]^+\qquad p^{\delta,+}(x)\rightarrow p^+(x):=[K_2+c_0x]^+
$$
uniformly on $\R^-$, where $[.]^+$ denotes the positive part. Taking the limit $\delta\rightarrow 0$ in \eqref{eq:pdelta-leqpdeltaleqpdelta+} yields
$$
x\leq 0\Rightarrow p^-(x)\leq p(x,y)\leq p^+(x).
$$
In particular
\begin{eqnarray*}
x<x_0:=-\frac{K_2}{c_0} & \Rightarrow & p(x,y)\leq p^+(x)=0,\\
x>x_1:=-\frac{K_1}{c_1} & \Rightarrow & p(x,y)\geq p^-(x)>0,
\end{eqnarray*}
and $p$ has a non-trivial interface of finite width $\Gamma:=\partial\{p>0\}\subset\{x_0\leq x\leq x_1\}$ as pictured in Figure \ref{fig:interface}.
\\
For any $y\in\T$ the quantity
\begin{equation}
I(y):=\displaystyle{\inf}(x\in\R,\quad p(x,y)>0)
\label{eq:definterface}
\end{equation}
is well defined because $p$ is nondecreasing in $x$, and by monotonicity there holds $p(x,y)>0\Leftrightarrow x>I(y)$. This function $I(.)$ is upper semi-continuous, since its hypograph
$$
\Big{\{}(x,y),\quad x\leq I(y)\Big{\}}=\Big{\{}(x,y),\quad p(x,y)\leq 0\Big{\}}=\Big{\{}(x,y),\quad p(x,y)=0\Big{\}}=D\setminus D^+
$$
is a closed set ($p$ is continuous).
\begin{figure}[!h]
\begin{center}
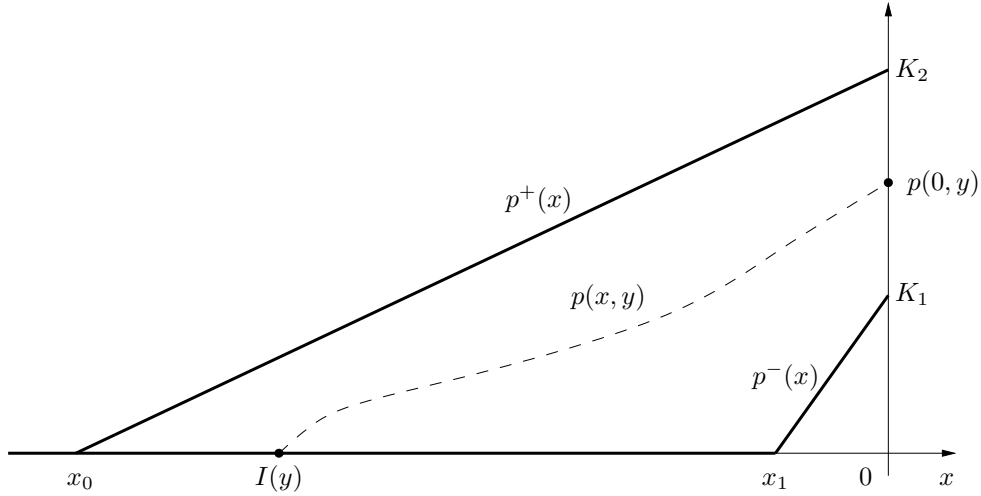
\end{center}
\caption{Existence and width of the free boundary.}
\label{fig:interface}
\end{figure}
\item
Assume that $y_0\in\T$ is such that $\underline{I}(y_0):=\displaystyle{\liminf_{y\rightarrow y_0}}<I(y_0)$; we prove by double inclusion that, if $\Gamma=\partial\{p>0\}$, then $\Gamma\cap\{y=y_0\}=[\underline{I}(y_0),I(y_0)]\times\{y=y_0\}$. We write for simplicity $\Gamma_0:=\Gamma\cap\{y=y_0\}$, and let us point out that by definition $p(x,y_0)=0$ holds for $x\leq I(y_0)$.
\begin{itemize}
\item
\fbox{
$\mathbf{\Gamma_0\subset[\underline{I}(y_0),I(y_0)]\times\{y=y_0\}}$
}\;
If $x_0>I(y_0)$ we have that $(x_0,y_0)\in D^+$, therefore $(x_0,y_0)\notin \Gamma=\overline{D^+}/D^+$ and thus $\Gamma_0\subset ]-\infty,I(y_0)]\times \{y=y_0\}$. If $x_0<\underline{I}(y_0)$, assume that there exists a sequence $(x_n,y_n)\rightarrow(x_0,y_0)$ such that $p(x_n,y_n)\in D^+$: by definition of $I(.)$ we have that $p(x_n,y_n)>0\Rightarrow x_n>I(y_n)$, and as a consequence $x_0\geq \displaystyle{\liminf_{y\rightarrow y_0}}I(y)=\underline{I}(y_0)$. This is impossible since we assumed $x_0<\underline{I}(y_0)$, hence $\Gamma_0\subset [\underline{I}(y_0),I(y_0)]\times\{y=y_0\}$.
\item
\fbox{
$\mathbf{\Gamma_0\supset[\underline{I}(y_0),I(y_0)]\times\{y=y_0\}}$
}\;
Choose any point $(x_0,y_0)\in [\underline{I}(y_0),I(y_0)]\times\{y=y_0\}$: since $p(x_0,y_0)=0$ and $\Gamma=\overline{D^+}/D^+$ we only need to build a sequence $(x_n,y_n)\rightarrow (x_0,y_0)$ such that $(x_n,y_n)\in D^+$. Let $y_n\rightarrow y_0$ be any sequence such that $I(y_n)\rightarrow \underline{I}(y_0)$. If $x_0=\underline{I}(y_0)$ define $x_n:=I(y_n)+1/n$: we have that $x_n>I(y_n)\Rightarrow p(x_n,y_n)>0$ hence $(x_n,y_n)\in D^+$, and clearly $(x_n,y_n)\rightarrow (\underline{I}(y_0),y_0)$. If now $x_0>\underline{I}(y_0)$, define $x_n:=x_0$: for $n$ large enough we have again $x_n>I(y_n)$ hence $(x_n,y_n)\in D^+$, and $(x_n,y_n)\rightarrow (x_0,y_0)$. Therefore $[\underline{I}(y_0),I(y_0)]\times\{y=y_0\}\subset\Gamma_0$.
\end{itemize}
\end{enumerate}
\end{proof}
We prove now the Lipschitz regularity stated in Proposition~\ref{prop:lipschitz}:
\begin{proof}
Under the non-degeneracy hypothesis $\left.p_x\right|_{D^+}\geq a>0$ we prove that the graph of $I(y)$ can be obtained as the uniform limit of the $\eps$-levelset of $p$ when $\eps\rightarrow 0^+$, and that these levelsets are Lipschitz uniformly in $\eps$.
\par
Let us recall that $\left.p\right|_{D^+}\in\mathcal{C}^{\infty}(D^+)$: the strict $x$ monotonicity and the Implicit Functions Theorem show that, for any $\eps>0$, the $\eps$-levelset of $p$ can be globally parametrized as a smooth hypersurface
$$
p(x,y)=\eps\quad\Leftrightarrow \quad x=I_{\eps}(y),
$$
where $I_{\eps}\in\mathcal{C}^{\infty}(\T)$. Moreover $\frac{dI_{\eps}}{dy}=-\cfrac{p_y}{p_x}$, and the non-degeneracy hypothesis combined with proposition \ref{prop:estimatepy} guarantee that
$$
\left|\frac{dI_{\eps}}{dy}\right|\leq C
$$
for some constant $C$ independent of $\eps$. By Arzel\`a-Ascoli Theorem we can assume, up to extraction, that $I_{\eps}(.)$ converges to some $J(.)$ uniformly on $\T$. This limit is of course Lipschitz, and we show below that $J(y)=I(y)$, where $I$ is defined as in Theorem \ref{theo:pdelta->p} ($p(x,y)>0\Leftrightarrow x>I(y)$).

By continuity we have that $p(I_{\eps}(y),y)=\eps\Rightarrow p(J(y),y)=0$. Chooss $x_0<J(y)$ and $\eps$ small enough: integrating $p_x\geq a>0$ from $x=I_{\eps}(y)<x_0$ to $x=x_0$ leads to $p(x_0,y)\geq \eps +a(x_0-I_{\eps}(y))$. Taking the limit $\eps\rightarrow 0^+$ yields $p(x_0,y)\geq a(x_0-J(y))>0$, and therefore
$$
\left.
\begin{array}{c}
 p(J(y),y)=0\\
x_0>J(y)\Rightarrow p(x_0,y)>0
\end{array}
\right\}\quad \Rightarrow \quad J(y)=\inf\Big{(}x,\quad p(x,y)>0\Big{)}=I(y)
$$
by definition \eqref{eq:definterface} of $I$. Thus $I=J$ is Lipschitz, and by continuity $\Gamma=\{x=I(y)\}$.
\end{proof}
\begin{prop}
The corresponding temperature variable $v=\left(\frac{m}{m+1}p\right)^{\frac{1}{m}}\in\mathcal{C}(D)$ solves the original equation $\Delta(v^{m+1})=(c+\alpha)v_x$ in the weak sense on the infinite cylinder $D$: for any test function $\Psi\in\mathcal{D}(D)$ with compact support $K\subset D$ we have that
$$
\displaystyle{\iint\limits_{K}v^{m+1}\Delta\Psi\mathrm{d}x\mathrm{d}y}+\displaystyle{\iint\limits_{K}(c+\a)v\Psi_x\mathrm{d}x\mathrm{d}y}=0.
$$
\label{prop:pweaksolution}
\end{prop}
\begin{proof}
We denote by $v^L$ and $v^{\delta}$ the temperature variable corresponding to our two successive approximations $p^L$ and $p^{\delta}$. Let $\Psi\in\mathcal{D}$ be any such test function with compact support $K\subset D$: let us recall that the finite cylinder grows in both direction, and consequently $K\subset D_L$ for $L$ large enough. $p^{L}>0$ was a smooth solution of $-mp\D p+(c+\a)p_x=|\nabla p|^2$ so that $v^{L}$ was a smooth solution of $\Delta(v^{m+1})-(c+\a)v_x=0$, and therefore
$$
\displaystyle{\iint\limits_{K}(v^{L})^{m+1}\Delta\Psi\mathrm{d}x\mathrm{d}y}+\displaystyle{\iint\limits_{K}(c+\a)v^L\Psi_x\mathrm{d}x\mathrm{d}y}=0.
$$
When $L\rightarrow +\infty$ the $\mathcal{C}^2_{loc}(D)$ convergence $p^{L}\rightarrow p^{\delta}$ implies the $\mathcal{C}^0_{loc}(D)$ convergence $v^L\rightarrow v^{\delta}$, hence
$$
\displaystyle{\iint\limits_{K}(v^{\delta})^{m+1}\Delta\Psi\mathrm{d}x\mathrm{d}y}+\displaystyle{\iint\limits_{K}(v^{\delta})(c+\a)v\Psi_x\mathrm{d}x\mathrm{d}y}=0.
$$
Using the $\mathcal{C}^0_{loc}(D)$ convergence $p^{\delta}\rightarrow p$ the integrals above finally pass to the limit $\delta\rightarrow 0$.
\end{proof}
%
%
\section{Behavior at infinity}
\label{section:linear}
We prove in this section that the behavior at infinity is not perturbed by the shear flow, compared to the classical PME traveling wave $p(x,y)=c\left[x-x_0\right]^+$. As mentioned above the results of this section are established directly for the final viscosity solution $p=\lim p^{\delta}$, but easily extend to the $\delta$-solutions.
\begin{theo}
$p(x,y)$ is planar and $x$-linear at infinity, with slope exactly equal to the speed:
$$
p_x(x,y)\sim c \qquad p_y(x,y)\rightarrow 0, \qquad p(x,y)\sim cx
$$
uniformly in $y$ when $x\rightarrow +\infty$.
\label{theo:plinear}
\end{theo}
We start by showing that $p(x,y)$ grows at least and at most linearly for two different slopes; using a Lipschitz scaling under which the equation is invariant, we will deduce that $p$ is exactly linear and that its slope is given by its speed $c>0$. This will be done by proving that in the limit of an infinite zoom-out $(x,y)\rightarrow(X,Y)$ the scaled solution $P(X,Y)$ converges to a weak solution the usual PME ($\alpha\equiv0$) which has a flat free boundary $X=0$ and is in-between two hyperplanes. By uniqueness for such weak solutions of the PME our solution will agree with the classical planar traveling wave $P(X,Y)=[cX]^+$, hence the slope for $p(x,y)$ at infinity.
%
%
\subsection{Minimal growth}
Since $p_x\leq c_1$ we have an upper bound at infinity $p\leq c_1x$; we show in this section that we also have a similar lower bound:
\begin{theo}
There exists $\underbar{C}>0$ such that
$$
x\geq 0\quad \Rightarrow \quad p(x,y)\geq \underbar{C}x.
$$
\label{theo:minlingrowth}
\end{theo}
Let us recall that we have pinned
$$
K_1\leq p(0,y)\leq K_2,\qquad K\leq \int \limits_{\T}p(0,y)\mathrm{d}y\leq K+C
$$
where $K_1\geq K-C\sqrt{K}$ and $K_2\leq K+C\sqrt{K}$. The constants $C$ above depend only on $m>0$ and the upper bound for the flow $c_1\geq c+\alpha(y)$, and $K>0$ can be chosen as large as required (see proof of proposition \ref{prop:pinning} for details).
\par
We will denote by
$$
O(x)=\displaystyle{\max_{y\in\T}p(x,y)}-\displaystyle{\min_{y\in\T}p(x,y)}
$$
the oscillations in the $y$ direction, which is a relevant quantity that we will need to control.
\begin{lem}
There exists a constant $C>0$ and a sequence $\left(x_n\right)_{n\geq 0}\in[n,n+1]$ such that
$$
O(x_n)\leq C\sqrt{\int\limits_{\T} p(n+1,y)dy}
$$
\label{lem:O(xn)leq}
\end{lem}
\begin{proof}
Integrating by parts $-mp\D p +(c+\alpha)p_x=|\nabla p|^2$ over $K_n=[n,n+1]\times \T$ we obtain
\begin{equation}
(m-1)\displaystyle{\iint\limits_{K_n}|\nabla p|^2\mathrm{d}x\mathrm{d}y}  +  m\displaystyle{\int\limits_{\T} pp_x(n,y)\mathrm{d}y}-m\displaystyle{\int\limits_{\T} pp_x(n+1,y)\mathrm{d}y}+\displaystyle{\iint\limits_{K_n}(c+\a)p_x\mathrm{d}x\mathrm{d}y}=0.
\label{eq:IBP_Kn}
\end{equation}
We distinguish again $m<1$ and $m>1$:
\begin{enumerate}
 \item 
If $m<1$ we use $pp_x(n+1,y)>0$, $0<c+\alpha\leq c_1$ and $0<p_x\leq c_1$ in \eqref{eq:IBP_Kn} to obtain
$$
(1-m)\displaystyle{\iint\limits_{K_n}|\nabla p|^2\mathrm{d}x\mathrm{d}y}\leq m\displaystyle{\int\limits_{\T} pp_x(n,y)\mathrm{d}y}+\displaystyle{\iint\limits_{K_n}(c+\a)p_x\mathrm{d}x\mathrm{d}y}\leq mc_1\displaystyle{\int\limits_{\T} p(n,y)\mathrm{d}y}+c_1^2.
$$
Choosing $K$ large enough we can assume by monotonicity that $c_1^2\leq mc_1\int p(n,y)\mathrm{d}y$, and therefore
$$
\displaystyle{\iint\limits_{K_n}|\nabla p|^2\mathrm{d}x\mathrm{d}y}\leq \frac{2mc_1}{1-m}\displaystyle{\int\limits_{\T} p(n,y)\mathrm{d}y}.
$$
If $x_n\in[n,n+1]$ is any point where $O(x)$ attains its minimum on this interval, then
$$
\begin{array}{cclcc}
O(x_n) & \leq & \int\limits_{n}^{n+1}O(x)dx & &\\
 & \leq & \int\limits_n^{n+1} \left(\int\limits_{\T}|p_y(x,y)|dy\right)dx & \leq & C\sqrt{ \iint\limits_{K_n}|\nabla p|^2dy}.
\end{array}
$$
Using our previous estimate and monotonicity we finally obtain
$$
 O(x_n) \leq  C\sqrt{\int\limits_{\T} p(n,y)dy} \leq  C\sqrt{\int\limits_{\T} p(n+1,y)dy},
$$
where $C$ depends only on $m$ and $c_1$.
\item
If $m>1$ we use $pp_x(n,y)>0$, $(c+\alpha)p_x>0$ and $p_x\leq c_1$ in \eqref{eq:IBP_Kn}, yielding
$$
(m-1)\displaystyle{\iint\limits_{K_n}|\nabla p|^2\mathrm{d}x\mathrm{d}y}\leq m\displaystyle{\int\limits_{\T} pp_x(n+1,y)\mathrm{d}y}\leq mc_1\displaystyle{\int\limits_{\T} p(n+1,y)\mathrm{d}y}.
$$
The rest of the computation is similar to the case $m<1$.
\end{enumerate}
\end{proof}
\begin{cor}
There exists $C>0$ such that
$$
x\geq 0\qquad \Rightarrow \qquad O(x)\leq C\sqrt{\int\limits_{\T} p(x,y)dy}.
$$
\label{cor:O(x)leq}
\end{cor}
\begin{proof}
Since $0<p_x \leq c_1$ the function $O(x)$ is clearly $2c_1$-Lipschitz, and by Lemma \ref{lem:O(xn)leq} ensures that
$$
O(x)\leq O(x_n)+2c_1\leq C\sqrt{\int\limits_{\T} p(n+1,y)dy}+2c_1
$$
for any $x\in[n,n+1]$. By monotonicity we can moreover assume that $2c_1\leq C\sqrt{\int p(n+1,y)dy}$ if $K$ is chosen large enough, and therefore
$$
O(x)\leq C\sqrt{\int\limits_{\T} p(n+1,y)dy}.
$$
For the same reason we can also assume that
$$
\int p(n+1,y)dy\leq \int p(x,y)dy+c_1\leq C \int p(x,y)dy,
$$
and combining with the previous inequality yields the desired result.
\end{proof}
\begin{prop}
For any $x\geq 0$ we have that
$$
\frac{d}{d x}\left(\int\limits_{\T}p^{\frac{m+1}{m}}(x,y)dy\right)=\frac{m+1}{m}\int\limits_{\T}(c+\alpha(y))p^{\frac{1}{m}}(x,y)dy.
$$
\label{prop:df/dxgeq...}
\end{prop}
\begin{proof}
We establish this equality for the uniformly elliptic solution $p^{\delta}\geq \delta$ up to a constant $C_{\delta}$, with $C_{\delta}\rightarrow 0$ when $\delta\rightarrow 0$. The equation for $p^{\delta}$ can be written in the divergence form
$$
\nabla\cdot\left(\left(p^{\delta}\right)^{\frac{1}{m}}\nabla p^{\delta}\right)=\left((c+\alpha)\left(p^{\delta}\right)^{\frac{1}{m}}\right)_x,
$$
and integrating by parts over $\Omega=[x_1,x_2]\times \T$ yields
$$
\int\limits_{\T}\left(p^{\delta}\right)^{\frac{1}{m}}p^{\delta}_x(x_2,y)dy-\int\limits_{\T}\left(p^{\delta}\right)^{\frac{1}{m}}p^{\delta}_x(x_1,y)dy=\int\limits_{\T}(c+\alpha)\left(p^{\delta}\right)^{\frac{1}{m}}(x_2,y)dy-\int\limits_{\T}(c+\alpha)\left(p^{\delta}\right)^{\frac{1}{m}}(x_1,y)dy
$$
for any $x_1<x_2$. As a consequence, the quantity
\begin{equation}
F(x):=\int\limits_{\T}\left(p^{\delta}\right)^{\frac{1}{m}}p^{\delta}_x(x,y)dy-\int\limits_{\T}(c+\alpha)\left(p^{\delta}\right)^{\frac{1}{m}}(x,y)dy\equiv C_{\delta}
\label{eq:F(x)=cst}
\end{equation}
is constant. Let us recall from proposition \ref{prop:pdelta(-infty)=delta} that $p^{\delta}(-\infty,y)=\delta$ uniformly in $y$, and also the uniform bounds $c_0\leq c+\alpha\leq c_1$ and $0<p^{\delta}_x\leq c_1$: taking the limit $x\rightarrow -\infty$ in \eqref{eq:F(x)=cst} leads to $C_{\delta}=\mathcal{O}\left(\delta^{\frac{1}{m}}\right)$.
\par
Fix any $x>0$: the strong $C^1_{loc}$ convergence $p^{\delta}\rightarrow p$ on $D^+=\{p>0\}$ is strong enough to take the limit $\delta\rightarrow$ in \eqref{eq:F(x)=cst}, which reads
$$
\int\limits_{\T}p^{\frac{1}{m}}p_x(x,y)dy-\int\limits_{\T}(c+\alpha)p^{\frac{1}{m}}(x,y)dy=0
$$
Finally, $p$ is is smooth for $x>0$ (because $p>0$), and the last equality above easily yields the desired differential equation.
\end{proof}
We can now prove the claimed minimal growth:
\begin{proof}(of theorem \ref{theo:minlingrowth}). Define $f(x):=\int\limits_{\T}p^{\frac{m+1}{m}}(x,y)dy$: Proposition \ref{prop:df/dxgeq...} reads
\begin{equation}
f'(x)= \frac{m+1}{m}\int\limits_{\T}(c+\alpha)p^{\frac{1}{m}}dy
\label{eq:f'(x)=}
\end{equation}
for $x>0$. By monotonicity $\int\limits_{\T}p(x,y)dy\geq\int\limits_{\T}p(0,y)dy=K$, and by corollary \ref{cor:O(x)leq} we control the oscillations $O(x)\leq C\sqrt{\int p(x,y)dy}$. Choosing $K$ large enough the oscillations of $p$ are small compared to its average along \textit{any} line $x=cst\geq 0$, hence
$$
\int\limits_{\T}(c+\alpha)p^{\frac{1}{m}}dy\geq c_0\int\limits_{\T}p^{\frac{1}{m}}dy\geq C\left(\int\limits_{\T}p^{\frac{m+1}{m}}dy\right)^{\frac{1}{m+1}}=Cf^{\frac{1}{m+1}}(x).
$$
This estimate combined with \eqref{eq:f'(x)=} leads to $f'(x)\geq Cf^{\frac{1}{m+1}}(x)$,
and integration yields
$$
f^{\frac{m}{m+1}}(x)\geq Cx.
$$
Finally, since we control the oscillations of $p$,
$$
p(x,y)\geq C\int\limits_{\T}p(x,y)dy\geq C\left(\int\limits_{\T}p^{\frac{m+1}{m}}(x,y)dy\right)^{\frac{m}{m+1}}\geq Cf^{\frac{m}{m+1}}(x)\geq \underline{C}x.
$$
\end{proof}

%
\subsection{Proof of Theorem \ref{theo:plinear}}
We start by estimating how fast $p$ becomes planar at infinity:
\begin{prop}
Let as before $O(x):=\displaystyle{\max_{y\in\T}\;p(x,y)}-\displaystyle{\min_{y\in\T}\;p(x,y)}$; there exists $C>0$ such that when $x\rightarrow +\infty$
$$
O(x)\leq \frac{C}{x}.
$$
\label{prop:p_planar_infinity}
\end{prop}
\begin{proof}
For $x$ large enough we know that $p(x,y)>0$ is smooth; $w:=\frac{m^2}{m+1}p^{\frac{m+1}{m}}$ is therefore smooth, and satisfies as before
$$
\D(w)=f,\qquad f=(c+\alpha)p^{\frac{1}{m}-1}p_x.
$$
We will first show that the $y$ oscillations of $w$ cannot blow too fast when $x\rightarrow +\infty$, and then deduce the desired planar behavior for $p$.
\par
The Fourier series
$$
w(x,y)=\sum_{n\in\mathbb{Z}}w_n(x)e^{2i\pi ny}
$$
is at least pointwise convergent, and for $n\neq 0$ we have that
\begin{equation}
-w_n''(x)+4\pi^2n^2w_n(x)=f_n(x),\qquad f_n(x):=-\displaystyle{\int\limits_{\T}f(x,y)e^{-2i\pi ny}\mathrm{d}y}.
\label{eq:fourier_wn}
\end{equation}
The oscillations of $w$ in the $y$ direction are completely described by its Fourier coefficients $w_n(x)$ for $n\neq 0$, in which case \eqref{eq:fourier_wn} is strongly coercive. This coercivity will allow us to control how fast $w_n(x)$ may grow when $x\rightarrow +\infty$, and therefore how much $w$ can oscillate.
\par
Since $p$ is at least and at most linear and $p_x, c+\alpha$ are bounded we control
\begin{equation}
|f_n|(x)\leq Cx^{\frac{1}{m}-1}
\label{eq:estimate_fn}
\end{equation}
uniformly in $n$. Moreover, taking real and imaginary parts of \eqref{eq:fourier_wn}, we may assume that $w_n(x)$, $f_n(x)$ are real and that $n=|n|\geq 0$.
\begin{itemize}
\item
We claim that there exists $C>0$ such that, for any $n\neq 0$ and $x\rightarrow +\infty$, there holds
\begin{equation}
|w_n(x)|\leq \frac{C}{n^2}x^{\frac{1}{m}-1}.
\label{eq:wn_leq_Cn...}
\end{equation}
Indeed, since $0\leq w=\frac{m^2}{m+1}p^{\frac{m+1}{m}}\leq Cx^{\frac{m+1}{m}}$, we have that
$$
|w_n|^2(x)\leq ||w(x,.)||^2_{L^2(\T)}\leq Cx^{2\frac{m+1}{m}}.
$$
As a consequence $w_n$ cannot have a component on the homogeneous solution $e^{+ 2\pi n x}$ of \eqref{eq:fourier_wn} for $n\neq 0$, and it is then easy to see that it is explicitly given by
\begin{equation}
 w_n(x)=e^{-2\pi n(x-x_0)}w_n(x_0)+e^{-2\pi nx}\int\limits_{x_0}^xe^{4\pi nz}\left(\int\limits_z^{+\infty}e^{-2\pi nt}f_n(t)dt\right)dz.
\label{eq:wn_explicit}
\end{equation}
Our claim \eqref{eq:wn_leq_Cn...} is then easily obtained manipulating this explicit formula, the computations involving several integrations by parts and the fact that $w_n(x_0)$ is rapidly decreasing in $n$ (since $w(x_0,.)\in\mathcal{C}^{\infty}(\T)$).
\item
As a consequence of \eqref{eq:wn_leq_Cn...}, the series
$$
w^{\perp}(x,y):=w(x,y)-\int\limits_{\T}w(x,y)dy=\sum_{n\neq 0}w_n(x)e^{2i\pi ny}
$$
is  uniformly convergent and $|w^{\perp}(x,y)|\leq Cx^{\frac{1}{m}-1}$. This clearly bounds the oscillations of $w$ when $x\rightarrow +\infty$ by
\begin{equation}
\displaystyle{\max_{y\in\T}\;w(x,y)}-\displaystyle{\min_{y\in\T}\;w(x,y)}\leq 2||w^{\perp}(x,.)||_{L^{\infty}(\T)}\leq C x^{\frac{1}{m}-1}.
\label{eq:oscillations_w}
\end{equation}
\end{itemize}
Translating the oscillations of $w$ in terms of those of $p=Cw^{\frac{m}{m+1}}$ leads to
\begin{eqnarray*}
O(x) & = & C\times \left(\displaystyle{\max_{y\in\T}\;w^{\frac{m}{m+1}}(x,y)}-\displaystyle{\min_{y\in\T}\;w^{\frac{m}{m+1}}(x,y)}\right)\\
 & \leq & C\left(\displaystyle{\min_{y\in\T}\;w(x,y)}\right)^{\frac{m}{m+1}-1}\left[\displaystyle{\max_{y\in\T}\;w(x,y)}-\displaystyle{\min_{y\in\T}\;w(x,y)}\right].
\end{eqnarray*}
Since $w=\frac{m^2}{m+1}p^{\frac{m+1}{m}}\geq Cx^{\frac{m+1}{m}}$ and $\frac{m}{m+1}-1=-\frac{1}{m+1}$, estimate \eqref{eq:oscillations_w} finally implies that
$$
O(x)\leq C\left(x^{\frac{m+1}{m}}\right)^{-\frac{1}{m+1}}\times Cx^{\frac{1}{m}-1}=\frac{C}{x}.
$$
\end{proof}
For any $\eps>0$ let us introduce the Lipschitz scaling
$$
 P^{\eps}(X,Y)=\eps p(x,y),\qquad (x,y)=\frac{1}{\eps}(X,Y);
$$
when $\eps\rightarrow 0^+$ this corresponds to zooming out on the whole picture. Uppercase letters will denote below the ``fast'' variables and functions, whereas lowercase will denote the ``slow'' ones. Since we want to zoom out it will be more convenient to consider below the cylinder $D=\R\times\T$ as a plane $\R^2$ with a $1$-periodicity condition for $p$ in the $y$ direction, corresponding to a plane with $\eps$-periodicity in $Y$ for $P^{\eps}$.
\\
\par
The proof of Theorem \ref{theo:plinear} relies on three key points: the first one is that the equation is invariant under this scaling. The second one is that, since the shear flow $\alpha(y)$ is $1$-periodic with mean zero, the corresponding flow $A^{\eps}(Y)=\alpha(Y/\eps)$ is $\eps$-periodic with mean zero in $Y$: Riemann-Lebesgue Theorem guarantees that $A^{\eps}\rightharpoonup 0$ in a weak sense when $\eps\rightarrow 0$, so that any limiting profile $P=\lim P^{\eps}$ will not ``see the flow'' and thus satisfy the usual PME $-mP\D P+(c+0)P_X=|\nabla P|^2$. Finally, proposition \ref{prop:p_planar_infinity} guarantees that the oscillations of $p$ in the $y$ direction decrease at infinity: zooming out, the limit $P$ will therefore be planar, $P_Y\equiv 0$.
\par
In the limit of this infinite zoom-out the scaled profile indeed converges:
\begin{prop}
Up to a subsequence we have $P^{\eps}(X,Y)\rightarrow P(X,Y)$ when $\eps\rightarrow 0^+$. The convergence is uniform on $\R^-\times\R$ and $\mathcal{C}^1_{loc}$ on $\R^{+*}\times\R$. Further:
\begin{enumerate}
\item $P$ is continuous on the whole plane and $P\equiv 0$ for $X\leq 0$
\item $0<\underbar{C}X\leq P(X,Y)\leq c_1 X$ for $X>0$, where $\underbar{C}>0$ is the constant in Theorem \ref{theo:minlingrowth} and $c_1\geq c+\alpha(y)$ is the upper bound for the flow.
\end{enumerate}
\label{prop:PepsCVP}
\end{prop}
\begin{proof}
We pinned the original solution $p$ such that $0\leq p(x,y)\leq K_2$ for $x\leq 0$, and this immediately implies that $P^{\eps}=\eps p\leq \eps K_2\rightarrow 0$ uniformly on the closed left half-plane $X\leq 0$. On the right half-plane $0<P^{\eps}_X(X,Y)=p_x(x,y)\leq c_1$ bounds $P^{\eps}$ from above as
\begin{equation}
P^{\eps}(X,Y)\leq P^{\eps}(0,Y)+c_1 X\leq K_2\eps + c_1 X,
\label{eq:Pepsleqc1X}
\end{equation}
and Theorem \ref{theo:minlingrowth} bounds $P^{\eps}$ away from zero
\begin{equation}
P^{\eps}(X,Y)=\eps p(X/\eps,Y/\eps) \geq \underbar{C}X.
\label{eq:PepsgeqcX}
\end{equation}
Let us recall that $p$ is a smooth classical solution on $D^+=\{p>0\}\supset \R^+\times\T$: for $\eps>0$ the rescaled profile $P^{\eps}$ is therefore a smooth classical solution of the rescaled equation
$$
-mP^{\eps}\Delta_{X,Y} P^{\eps} +\left[c+A^{\eps}(Y)\right]P^{\eps}_X=\left|\nabla_{X,Y}P^{\eps}\right|^2,\qquad A^{\eps}(Y)=\alpha(Y/\eps),
$$
at least for $X>0$.

Using our previous interior elliptic $L^q$ regularity argument for
$$
W^{\eps}:=\frac{m^2}{m+1}\left(P^{\eps}\right)^{1+\frac{1}{m}}, \qquad F^{\eps}:= (c+A^{\eps})\left(P^{\eps}\right)^{\frac{1}{m}-1}P^{\eps}_X, \qquad \Delta W^{\eps}=F^{\eps},
$$
we obtain as before an estimate
$$
||W^{\eps}||_{W^{2,q}(\mathcal{B_1})}\leq C
$$
uniformly in $\eps$ on any ball $\mathcal{B}_1\subset\R^{+*}\times \R$ of radius $1$ and for $q>d=2$ (see proof of Theorem \ref{theo:pL->p} for details). It is here important that $P^{\eps}$ is bounded away from zero uniformly in $\eps$ for $X>0$, see again proof of Theorem~\ref{theo:pL->p} (in particular the case $m<1$). By compactness $W^{2,q}\subset\subset \mathcal{C}^1$ on bounded balls ($q>d=2$) and moving the center of the ball $\mathcal{B}_1$ we may assume, up to extraction of a subsequence, that
$$
W^{\eps}\longrightarrow W\quad\text{in }\mathcal{C}^1_{loc}(\R^{+*}\times\R).
$$
Since we took care to step out of the zero set uniformly in $\eps$, this convergence easily translates into
$$
P^{\eps}\overset{\mathcal{C}^1_{loc}(\R^{+*}\times\R)}{\longrightarrow} P,
$$
and $P$ is continuous on $\R^{+*}\times\R$ as a locally uniform limit of continuous functions. Taking the limit $\eps\rightarrow 0$ for $X>0$ in $\underbar{C}X\leq P^{\eps}(X,Y)\leq K_2\eps+c_1 X$ we obtain
$$
X>0\quad \Rightarrow \quad \underbar{C}X\leq P(X,Y)\leq c_1 X
$$
as claimed, which gives as a by product the continuity along $X=0$ (let us recall that $P\equiv 0$ on the left half-plane).
\end{proof}
\begin{rmk}
No higher regularity can be obtained with this interior elliptic regularity argument: $\mathcal{C}^2$ convergence would require for example $W^{3,q}$ estimates, involving $\nabla_{(X,Y)}F^{\eps}$ which contains the singular derivative $\partial_Y A^{\eps}=\frac{1}{\eps}\partial_y \alpha$.
\end{rmk}%
As usual we need to determine the limiting equation satisfied by the limiting profile in some sense:
\begin{prop}
The limiting function $P$ solves the PME
$$
-mP\Delta_{(X,Y)} P+c P_X=|\nabla_{(X,Y)} P|^2
$$
in the weak sense on the whole plane.
\label{prop:WsolutionPME}
\end{prop}
\begin{proof}
By definition of weak solutions we want to prove that, for any test function $\Phi(X,Y)$ with compact support $K\subset \R^2$, the corresponding temperature $V(X,Y):=\left(\frac{m}{m+1}P(X,Y)\right)^{\frac{1}{m}}$ satisfies
$$
I:=\displaystyle{\iint\limits_{K}V^{m+1}\Delta \Phi\mathrm{d}X\mathrm{d}Y}+\displaystyle{\iint\limits_{K}cV\Phi_X\mathrm{d}X\mathrm{d}Y}=0
$$
(note that the shear flow $A^{\eps}(Y)\leftrightarrow \alpha(y)$ disappeared in the advection term). Let us recall from Proposition \ref{prop:pweaksolution} that $p$ was a weak solution on the whole plane, and that the equation is invariant under Lipschitz scaling: for any $\eps>0$ the scaled temperature $V^{\eps}$ therefore satisfies
\begin{equation}
I(\eps):=\displaystyle{\iint\limits_{K}\left(V^{\eps}\right)^{m+1}\Delta \Phi\mathrm{d}X\mathrm{d}Y}+\displaystyle{\iint\limits_{K}(c+A^{\eps})V^{\eps}\Phi_X\mathrm{d}X\mathrm{d}Y}=0;
\label{eq:Pepsweaksolution}
\end{equation}
the problem is as usual to take the limit in this formulation.
\begin{itemize}
\item
If $K\subset \R^{-*}\times\R$ this limit is straightforward: $(c+A^{\eps})$ is uniformly bounded ($c_0\leq c+A^{\eps}\leq c_1$), and according to Proposition \ref{prop:PepsCVP} $V^{\eps}=\left(\frac{m}{m+1}P^{\eps}\right)^{\frac{1}{m}}\rightarrow 0$ uniformly on $K$.
\item
If $K\subset \R^{+*}\times\R$ the limit $V$ is positive so there is no such trivial convergence; it is convenient to split \eqref{eq:Pepsweaksolution} in three parts $I=I_1+I_1+I_3=0$, with
$$
\begin{array}{ccl}
I_1(\eps) & := & \displaystyle{\iint\limits_{K}\left(V^{\eps}\right)^{m+1}\Delta \Phi\mathrm{d}X\mathrm{d}Y},\\
I_2(\eps) & := & c\displaystyle{\iint\limits_{K}V^{\eps}\Phi_X\mathrm{d}X\mathrm{d}Y},\\
I_3(\eps) & := & \displaystyle{\iint\limits_{K}A^{\eps}V^{\eps}\Phi_X\mathrm{d}X\mathrm{d}Y}.
\end{array}
$$
The $\mathcal{C}^0_{loc}$ convergence $P^{\eps}\rightarrow P$ shows that $I_1$ and $I_2$ immediately pass to the limit. To deal with $I_3$ we compute with Fubini Theorem
$$
I_3(\eps)  =  \displaystyle{\iint\limits_{K}A^{\eps}V^{\eps}\Phi_X\mathrm{d}X\mathrm{d}Y} = \displaystyle{\int\limits_{\R}A^{\eps}(Y)\underbrace{\left(\int\limits_{\R} V^{\eps}(X,Y)\Phi_X(X,Y)\mathrm{d}X\right)}_{:=\Psi^{\eps}(Y)}\mathrm{d}Y};
$$
since $\Phi$ has compact support and $V^{\eps}\rightarrow V$ uniformly on $K$ we deduce that $\Psi^{\eps}(Y)\rightarrow \Psi(Y)$ uniformly on $\R$. $\Psi^{\eps}$ and $\Psi$ have both compact support: the convergence $\Psi^{\eps}\rightarrow\Psi$ therefore also holds in $L^1(\R)$, and by Riemann-Lebesgue Theorem $A^{\eps}\rightharpoonup 0$ weakly in $L^1(\R)$ (let us recall that $A^{\eps}(Y)$ is $\eps$ periodic with mean zero). $I_3(\eps)$ is therefore a dual pairing $I_3(\eps)=\langle A^{\eps},\Psi^{\eps}\rangle_{(L_1',L1)}$ of a weakly converging sequence with a strongly convergent one: hence the limit $I_3(\eps)\rightarrow 0$.
\item
If $K\cap\{X=0\}\neq\emptyset$ the convergence is more delicate because $K$ crosses the free boundary and we do not have uniform convergence $V^{\eps}\rightarrow V$ on $K$; however since $P(0,Y)=0$ both $V$ and $V^{\eps}$ have to be small on a neighborhood of $K\cap\{X=0\}$. For small $r>0$ we prove that there exists $\eps_0>0$ such that for all $\eps\leq \eps_0$ there holds $|I-I(\eps)|\leq r$.
\par
For $\eta>0$ to be chosen later let us define the partition
$$
K=\underbrace{\Big{(}K\cap\{X<-\eta\}\Big{)}}_{:=K^-}\cup\underbrace{\Big{(}K\cap\{|X|\leq\eta\}\Big{)}}_{:=K^{\eta}}\cup\underbrace{\Big{(}K\cap\{X>+\eta\}\Big{)}}_{:=K^+};
$$
$K^{\eta}$ is a striped $\eta$-neighborhood of $K\cap\{X=0\}$. On $K^{\pm}$ we already proved that $I_1,I_2,I_3$ converge: we only have to cope with the contribution from $K^{\eta}$, and it is clearly enough to prove separately
\begin{equation}
\begin{array}{ccc}
\displaystyle{\iint\limits_{K^{\eta}}\left|\left(V^{\eps}\right)^{m+1}-V^{m+1}\right|.|\Delta \Phi|\mathrm{d}X\mathrm{d}Y} & \leq & \frac{r}{3}\\
c\displaystyle{\iint\limits_{K^{\eta}}|V^{\eps}-V|.|\Phi_X|\mathrm{d}X\mathrm{d}Y} &\leq & \frac{r}{3}\\
\displaystyle{\iint\limits_{K^{\eta}}\left|A^{\eps}\right|.|V^{\eps}-V|.|\Phi_X|\mathrm{d}X\mathrm{d}Y} &\leq & \frac{r}{3}
\end{array}.
\label{eq:intsmall}
\end{equation}
Let us recall the previous bounds for the pressure variables, derived from the scaling and the Lipschitz estimate in the $X$ direction:
$$
\begin{array}{ccl}
-\eta\leq X\leq 0 & : &\left\{
\begin{array}{l} 0\leq P^{\eps}\leq K_2\eps\\
P\equiv 0
\end{array}\right.
\\
0\leq X\leq \eta & : & \left\{
\begin{array}{l} 0\leq P^{\eps}(X,Y)\leq P^{\eps}(0,Y)+c_1X\leq K_2\eps+c_1X\\
P\leq c_1X
\end{array}\right.
\end{array}
.
$$
Choosing $\eta$ and $\eps$ small, any positive power of the pressures $P^{\eps},P$ can clearly be made as small as required on $K^{\eta}$; this is also true for any positive power of the corresponding temperatures $V^{\eps},V$ (being themselves positive powers of the pressure), and all the terms $|\Delta\Phi|,|\Phi_x|,\left|A^{\eps}\right|$ are bounded uniformly in $\eps$: we complete the proof using the celebrated triangular inequality in the integrals \eqref{eq:intsmall}.
\end{itemize}
\end{proof}
We can now finally prove Theorem \ref{theo:plinear}:
\begin{proof}
By Proposition \ref{prop:PepsCVP} and up to extraction we have that $P^{\eps}\rightarrow P$ uniformly on $\R^{-}\times\R$ and locally in $\mathcal{C}^1(\R^{+*}\times\R)$; the corresponding temperature $V\geq0$ is a weak solutions of the stationary PME $-\Delta_{X,Y}\left(V^{m+1}\right)+cV_X=0$ (previous proposition) and has a flat free boundary $X=0$ separating $P\equiv 0$ to the left from $P>0$ to the right, where it is in-between two hyperplanes $\underline{C}X\leq P(X,Y)\leq c_1X$. Moreover, proposition \ref{prop:p_planar_infinity} shows that the limiting profile is planar, $\partial_Y P\equiv0$. Indeed, for fixed $X_0>0$ and any $Y_1,Y_2$, we have for $\eps$ small enough
\begin{eqnarray*}
\left|P^{\eps}(X_0,Y_1)-P^{\eps}(X_0,Y_2)\right| & = & \eps \left|p(X_0/\eps,Y_1/\eps)-p(X_0/\eps,Y_2/\eps)\right|\\
   & \leq & \eps\times C\left(\frac{X_0}{\eps}\right)^{-\frac{1}{m}}\\
 & \leq & \eps^{1+\frac{1}{m}}\frac{C}{X_0^{\frac{1}{m}}};
\end{eqnarray*}
taking the limit $\eps\rightarrow 0$ yields $|P(X_0,Y_2)-P(X_0,Y_2)|=\displaystyle{\lim_{\eps\rightarrow 0}}|P^{\eps}(X_0,Y_2)-P^{\eps}(X_0,Y_2)|=0$ for any $x_0>0$ and $Y_1,Y_2$. It is well known that there exists only one such planar solution of the PME, which is the standard planar traveling wave
$$
P(X,Y)=[cX]^+.
$$
Since the limit is unique the whole sequence actually converges, $\displaystyle{\lim_{\eps\rightarrow 0}}\;P^{\eps}=P$: for any $x_{\eps}=\frac{1}{\eps}\rightarrow +\infty$ the $\mathcal{C}^0_{loc}$ convergence $P^{\eps}(X,Y)\rightarrow[cX]^+$ on $\R^{+*}\times \R$ shows that
$$\begin{array}{cclcc}
\displaystyle{\max_{y\in\T}\left|p\left(x_{\eps},y\right)-cx_{\eps}\right|} & = & \displaystyle{\max_{Y\in[0,\eps]}\left|\frac{1}{\eps} P^{\eps}\left(\eps x_{\eps},Y\right)-cx_{\eps}\right|} & & \\
 & = & x_{\eps}\;\displaystyle{\max_{Y\in[0,\eps]}\left| P^{\eps}\left(1,Y\right)-P(1,Y)\right|} & = & o(x_{\eps}),
\end{array}$$
which means precisely $p(x,y)\sim cx$ uniformly in $y$ when $x\rightarrow +\infty$. Using now the stronger $\mathcal{C}^1_{loc}$ convergence for $X>0$ we finally obtain
$$
\begin{array}{ccccccc}
\displaystyle{\max_{y\in\T}\left|p_x\left(x_{\eps},y\right)-c\right|} & = & \displaystyle{\max_{Y\in[0,\eps]}\left|P^{\eps}_X\left(1,Y\right)-P_X(1,Y)\right|} & = & \underset{\eps\rightarrow 0}{o}(1) & \quad \Rightarrow \quad & p_x\sim c\\
\displaystyle{\max_{y\in\T}\left|p_y\left(x_{\eps},y\right)-0\right|} & = & \displaystyle{\max_{Y\in[0,\eps]}\left|P^{\eps}_Y\left(1,Y\right)-P_Y(1,Y)\right|} & = & \underset{\eps\rightarrow 0}{o}(1) & \quad \Rightarrow \quad & p_y\rightarrow 0.
\end{array}
$$
\end{proof}
\subsection{Asymptotic expansion at infinity}
We have shown that $p(x,y)\sim cx$ uniformly in $y$ when $x\rightarrow +\infty$. In this Section we strengthen this estimate and derive and asymptotic expansion
$$
p(x,y)=cx+q(x,y)
$$
with $W^{1,\infty}$ estimates on  $q$ as $x \to +\infty$.
\\
\par
For any function $f(x,y)$ periodic in the $y$ direction, we denote the average (the projection onto constants in $L^2(\T))$ by 
$$
\langle f \rangle(x):=\int\limits_{\T}f(x,y)dy.
$$
The orthogonal projection onto functions with mean zero is denoted by
$$
f^{\perp}(x,y):=f(x,y)-\langle f \rangle(x).
$$
The $x$ derivative commutes with both these projectors, $\frac{d}{dx}\langle f \rangle=\langle f_x \rangle$ and $(f_x)^{\perp}=(f^{\perp})_x$. The ansatz $p(x,y)=cx+q(x,y)$ gives 
$$
\langle p \rangle(x)=cx+\langle q \rangle(x),\qquad p^{\perp}(x,y)=q^{\perp}(x,y),
$$
and $ q,\langle q\rangle,q^{\perp}$ are $o(x)$. The main result of this section is
\begin{theo}
When $x\rightarrow +\infty$, we have that:
\begin{enumerate}
 \item 
For any $m\neq 1$ the correction $q(x,y)$ becomes planar: there exists $C>0$ such that
$$
|q^{\perp}|(x,y)+|\nabla q^{\perp}|(x,y)\leq \frac{C}{x}.
$$
\item
Assume in addition that $1<m\notin \mathbb{N}^*$, and let $N=[m]$: there exists a finite sequence $q_1,...,q_N\in\R$ and some $q^*\in\R$ such that
$$
q(x,y)=x\left(q_1x^{-\frac{1}{m}}+q_2x^{-\frac{2}{m}}+...+q_Nx^{-\frac{N}{m}}\right) +q^*+o(1).
$$
\end{enumerate}
\label{theo:asymptotic_expansion}
\end{theo}
The orthogonal projection $p^{\perp}(x,y)$ is controlled by the oscillations in the $y$ direction $|p^{\perp}(x,y)|\leq O(x)=\displaystyle{\max_{y\in\T}\;p(x,y)}-\displaystyle{\min_{y\in\T}\;p(x,y)}$, and Proposition~\ref{prop:p_planar_infinity} therefore implies that
\begin{equation}
|q^{\perp}|(x,y)=|p^{\perp}|(x,y)\leq \frac{C}{x}
\label{eq:estimate_qperp}
\end{equation}
when $x\rightarrow +\infty$.
\\
\par
We prove the first estimate of the Theorem as a separate Proposition.
\begin{prop}
There exists $C>0$ such that
$$
|q^{\perp}(x,y)|+|\nabla q^{\perp}(x,y)|\leq \frac{C}{x}.
$$
\label{prop:estimate_qperp_C1}
\end{prop}
Let us stress that this statement holds for any $m$, although we will specifically consider $m>1$ in the sequel.
\begin{proof}
By \eqref{eq:estimate_qperp} we already control $|q^{\perp}|$, and it is enough to control its gradient. 
The equation for $p$ reads
\begin{equation}
\D p=\frac{(c+\a)p_x}{mp}-\frac{|\nabla p|^2}{mp},
\label{eq:laplace_p}
\end{equation}
and when $x\rightarrow +\infty$ we know that $\nabla p\rightarrow (c,0)$ and $p\sim cx$ uniformly in $y$: as a consequence
$|\D p|\leq \frac{C}{x}$. Averaging in $y$ yields $|\langle p \rangle ''|\leq \frac{C}{x}$, and therefore
$$
|\D (q^{\perp})|=|\D (p^{\perp})|=\left|\D p- \langle  p  \rangle  ''\right|\leq \frac{C}{x}.
$$
Choose now $x_0$ large and $y_0\in\T$, and denote by $\mathcal{B}_1$ the ball of radius $1$ centered at $(x_0,y_0)$. As discussed above there exists $C>0$ such that, if $x_0$ is chosen large enough,
$$
(x,y)\in\mathcal{B}_1\qquad \Rightarrow \qquad\left\{
\begin{array}{c}
|q^{\perp}|(x,y)\leq \frac{C}{x_0}\\
|\D q^{\perp}|(x,y)\leq \frac{C}{x_0}
\end{array}
\right..
$$
The constants above depend on the radius of the ball $R=1$ but not on its center. Finally, the classical elliptic theory for Poisson equation on a ball controls the gradient at the center by $|\nabla q^{\perp}|(x_0,y_0)\leq C\left(||q^{\perp}||_{L^{\infty}(\mathcal{B}_1)}+||\Delta q^{\perp}||_{L^{\infty}(\mathcal{B}_1)}\right)$, with $C$ depending only on the radius of the ball.
\end{proof}
As a corollary, we have that
\begin{lem}
 If $m>1$, there exists $\lambda\in\R$ such that
\begin{equation}
\langle  q  \rangle'(x)=\cfrac{\lambda}{\left(cx+\langle q\rangle\right)^{\frac{1}{m}}}+\mathcal{O} \left(\frac{1}{x^2}\right)
\label{eq:EDO_q'=}
\end{equation}
holds when $x\rightarrow +\infty$.
\label{lem:EDO_<q>}
\end{lem}
This technical result will later allow us to establish the asymptotic expansion $q=x(...)$ stated in Theorem \ref{theo:asymptotic_expansion}.
\begin{proof}
Equation\eqref{eq:laplace_p} with $p(x,y)=cx+q(x,y)$ leads to
\begin{equation}
 m\D q=\frac{(\a -c)q_x}{cx+q}-\frac{|\nabla q|^2}{cx+q}+\frac{c\alpha}{cx+q}.
\label{eq:laplace_q}
\end{equation}
By proposition \ref{prop:estimate_qperp_C1} we control $|q^{\perp}|=\mathcal{O}\left(1/x\right)$, and it is easy to expand
$$
\frac{1}{cx+q}  =  \frac{1}{cx+ \langle  q  \rangle  +q^{\perp}}
    =  \frac{1}{cx+ \langle  q  \rangle}\left(1-\frac{q^{\perp}}{cx+\langle q \rangle}+\mathcal{O}\left(\frac{1}{x^4}\right)\right).
$$
This expansion allows us to estimate separately the three terms in the right-hand side of \eqref{eq:laplace_q}, and in particular their average in $y$.
\begin{itemize}
 \item
The first one is
$$
\begin{array}{ccl}
A(x,y) :=  \frac{1}{cx+q}(\a -c)q_x  & = & -\frac{c}{cx+\langle q\rangle}\langle q\rangle'+\frac{\left(\alpha q^{\perp}\right)_x}{cx+\langle q\rangle}-\frac{\langle q\rangle'}{\left(cx+\langle q\rangle\right)^2}\left(\alpha q^{\perp}\right)  \\
 & & +\underbrace{\frac{\langle q\rangle'}{cx+\langle q\rangle}\alpha-\frac{c}{cx+\langle q\rangle}\left(q_x\right)^{\perp}}_{\text{purely orthogonal}} +\underbrace{\mathcal{O}\left(\frac{1}{x^3}\right)}_{\text{lower order}}.
\end{array}
$$
Averaging in $y$ then yields
\begin{equation}
 \langle A\rangle(x)=-\frac{c}{cx+\langle q\rangle}\langle q\rangle'+\frac{\langle\alpha q^{\perp}\rangle'}{cx+\langle q\rangle}-\frac{\langle q\rangle'}{\left(cx+\langle q\rangle\right)^2}\langle\alpha q^{\perp}\rangle+\mathcal{O}\left(\frac{1}{x^3}\right)
\label{eq:average_A}
\end{equation}
\item
We expand the second one as
$$
\begin{array}{ccl}
B(x,y): =  \frac{1}{cx+q}|\nabla q|^2 & = & \frac{1+\mathcal{O}\left(\frac{1}{x^2}\right)}{cx+\langle q\rangle}\Big{[}(\langle q\rangle')^2+\left|\nabla q^{\perp}\right|^2+2\langle q\rangle'q_x^{\perp}\Big{]} \\
 & = & \frac{\langle q\rangle'}{cx+\langle q\rangle'}\langle q\rangle'+\underbrace{\frac{2\langle q\rangle'}{cx+\langle q\rangle}\left(q_x\right)^{\perp}}_{\text{purely orthogonal}}+\underbrace{\mathcal{O}\left(\frac{1}{x^3}\right)}_{\text{lower order}},
\end{array}
$$
and averaging leads to
\begin{equation}
 \langle B\rangle(x)=\frac{\langle q\rangle'}{cx+\langle q\rangle'}\langle q\rangle'+\mathcal{O}\left(\frac{1}{x^3}\right).
\label{eq:average_B}
\end{equation}
\item
The last term is
$$
\begin{array}{ccl}
C(x,y) :=  \frac{1}{cx+q}c\alpha  & = & -\frac{c}{\left(cx+\langle q\rangle\right)^2}\left(\alpha q^{\perp}\right)+\underbrace{\frac{c}{cx+\langle q\rangle}\alpha}_{\text{purely orthogonal}}+\underbrace{\mathcal{O}\left(\frac{1}{x^5}\right)}_{\text{lower order}},
\end{array}
$$
and finally
\begin{equation}
 \langle C\rangle(x)=-\frac{c}{\left(cx+\langle q\rangle\right)^2}\langle\alpha q^{\perp}\rangle+\mathcal{O}\left(\frac{1}{x^5}\right).
\label{eq:average_C}
\end{equation}
Averaging \eqref{eq:laplace_q} in $y$ reads $m\langle q\rangle''(x)=\langle A\rangle(x)-\langle B\rangle(x)+\langle C\rangle(x)$: taking advantage of \eqref{eq:average_A}-\eqref{eq:average_B}-\eqref{eq:average_C} and rearranging, we obtain
$$
m \langle q\rangle''+\frac{c+\langle q\rangle'}{cx+\langle q\rangle}\langle q\rangle'=\left(\frac{\langle\alpha q^{\perp}\rangle}{cx+\langle q\rangle}\right)'+\mathcal{O}\left(\frac{1}{x^3}\right).
$$
Multiplying by the integrating factor $\left(cx+\langle q\rangle\right)^{\frac{1}{m}}$ yields
\begin{equation}
\left(\left(cx+\langle q\rangle\right)^{\frac{1}{m}}\langle q\rangle '\right)'=\frac{\left(cx+\langle q\rangle\right)^{\frac{1}{m}}}{m}\left(\frac{\langle\alpha q^{\perp}\rangle}{cx+\langle q\rangle}\right)'+\mathcal{O}\left(x^{\frac{1}{m}-3}\right).
\label{eq:EDO_qaverage_factor}
\end{equation}
If $f(x):=\frac{\left(cx+\langle q\rangle\right)^{\frac{1}{m}}}{m}\left(\frac{\langle\alpha q^{\perp}\rangle}{cx+\langle q\rangle}\right)'$ denotes the first term in the right-hand side above, an integration by parts combined with $|q^{\perp}|\leq C/x\Rightarrow\left|\langle\alpha q^{\perp}\rangle\right|\leq C/x$ allows us to show that $f$ is integrable at infinity and that
$$
\int\limits_x^{+\infty}f(z)dz=\mathcal{O}\left(x^{\frac{1}{m}-2}\right).
$$
This is precisely where we used the technical assumption $m>1$: otherwise this term may not be integrable at infinity.
\par
Equation \eqref{eq:EDO_qaverage_factor} can therefore be integrated from $x$ to $+\infty$: there is a $\lambda\in\R$ such that
$$
\left(cx+\langle q\rangle\right)^{\frac{1}{m}}\langle q\rangle '-\lambda=-\int\limits_{x}^{+\infty}\left[f(z)+\mathcal{O}\left(z^{\frac{1}{m}-3}\right)\right]\mathrm{d}z=\mathcal{O}\left(x^{\frac{1}{m}-2}\right),
$$
and we conclude the proof dividing by $\left(cx+\langle q\rangle\right)^{\frac{1}{m}}\sim Cx^{\frac{1}{m}}$.
\end{itemize}
\end{proof}
We finally prove Theorem~\ref{theo:asymptotic_expansion}.
\begin{proof}
The first item is stated in proposition \ref{prop:estimate_qperp_C1}. Regarding the second item, let us recall that $q=\langle q\rangle+q^{\perp}$ and that $|q^{\perp}|+|\nabla q^{\perp}|\leq C/x$: our statement is actually that the asymptotic expansion holds for $\langle q\rangle$ instead of $q$, since the transversal part $|q^{\perp}|$ is negligible when $x\rightarrow +\infty$.

Let us recall from Lemma~\ref{lem:EDO_<q>} that $\langle q\rangle(x)$ satisfies
\begin{equation}
\langle q\rangle'=\frac{\lambda}{\left(cx+\langle q\rangle\right)^{\frac{1}{m}}}+\mathcal{O}\left(\frac{1}{x^2}\right)
\label{eq:EDO_q_remind}
\end{equation}
for some $\lambda\in\R$. If $\lambda=0$ then $\langle q\rangle'$ is integrable and our statement immediately holds with $q_1=...=q_N=0$.
\par
If $\lambda\neq 0$ \eqref{eq:EDO_q_remind} with $cx+\langle q\rangle\sim cx$ yields $\langle q\rangle '\sim \lambda_1/x^{\frac{1}{m}}$, which is not integrable if $m>1$: integrating therefore yields $\langle q\rangle\sim q_1x^{1-\frac{1}{m}}$. Injecting this equivalent into \eqref{eq:EDO_q_remind} leads to
$$
\langle q\rangle'=\frac{\lambda}{\left(cx+q_1x^{1-\frac{1}{m}}+o\left(x^{1-\frac{1}{m}}\right)\right)^{\frac{1}{m}}}+\mathcal{O}\left(\frac{1}{x^2}\right).
$$
Expanding the quotient in Taylor series at order two in powers of $x^{-\frac{1}{m}}$ yields now
$$
\langle q\rangle '= \lambda_1x^{-\frac{1}{m}}+\lambda_2x^{-\frac{2}{m}}+o\left(x^{-\frac{2}{m}}\right)+\mathcal{O}\left(\frac{1}{x^2}\right),
$$
and integrating
$$
\langle q\rangle=x\left(q_1x^{-\frac{1}{m}}+q_2x^{-\frac{2}{m}}\right)+o\left(x^{-\frac{2}{m}}\right).
$$
Injecting again into \eqref{eq:EDO_q_remind} yields the next order, and so forth: by induction one shows that
\begin{equation}
\begin{array}{c}
\langle q\rangle=x\left(q_1x^{-\frac{1}{m}}+...+q_{k-1}x^{-\frac{k-1}{m}}\right) +o\left(x^{1-\frac{k-1}{m}}\right)\\
 \Downarrow \\
\langle q\rangle'=\lambda_1x^{-\frac{1}{m}}+...+\lambda_kx^{-\frac{k}{m}}+o\left(x^{-\frac{k}{m}}\right)+O\left(x^{-2}\right).
\end{array}
\label{eq:laurent_series_induction}
\end{equation}
\begin{itemize}
 \item 
As long as $k\leq N=[m]<m$ the last term $\lambda_kx^{-\frac{k}{m}}$ in the expansion of $\langle q\rangle'$ above is not integrable, and we may continue the induction
$$
\langle q\rangle'=\lambda_1x^{-\frac{1}{m}}+\lambda_kx^{-\frac{k}{m}}+o\left(x^{-\frac{k}{m}}\right)+O\left(x^{-2}\right) \quad \Rightarrow \quad \langle q\rangle=x\left(q_1x^{-\frac{1}{m}}+...+q_{k}x^{-\frac{k}{m}}\right) +o\left(x^{1-\frac{k}{m}}\right).
$$
\item
If now $k=N+1=[m]+1>m$, the terms $\lambda_kx^{-\frac{k}{m}}+o\left(x^{-\frac{k}{m}}\right)+O\left(x^{-2}\right)$ in \eqref{eq:laurent_series_induction} are integrable: integrating one last time we obtain as desired
$$
\langle q\rangle =x\left(q_1x^{-\frac{1}{m}}+...+q_{N}x^{-\frac{N}{m}}\right) +q^*+o(1),
$$
where $q^*$ is the constant of integration.
\end{itemize}
\end{proof}
\begin{rmk}
Let us stress that the condition $m\notin \mathbb{N}$ is purely technical. If $m=[m]=N$ is integer we may obtain at some point $\langle q\rangle'=\lambda_1x^{-\frac{1}{m}}+...+\frac{\lambda_{N}}{x}+..$ in the induction above. This would yield of course a logarithmic term, which would have to be properly taken into account. An asymptotic expansion could be obtained nonetheless, but the resulting computations would be long and not very insightful.
\label{rmk:m_noninteger}
\end{rmk}
\section{Uniqueness}
\label{section:uniqueness}
In this section we prove that the wave profiles of $\delta$-solutions are unique for given $\delta>0$ (and of course up to $x$-translations). Since we defined viscosity solutions as limits of $\delta$-solutions when $\delta\rightarrow 0^+$, uniqueness of viscosity solutions should follow. In order to keep this paper in a reasonable length we will not take this limit, which would require significant amount of technical work: one should indeed build a family of $\delta$-solutions $\left(p^{\delta}\right)_{\delta>0}$ such that the whole sequence converges to a non-trivial viscosity solution $p$ when $\delta\rightarrow 0^+$ (so far we only extracted a subsequence).
\par
Let us point out that all the results in Section \ref{section:linear} are stated for the final viscosity solution $p=\lim\;p^{\delta}$, but easily extend to the $\delta$-solutions for $\delta>0$. Through this whole section we fix $\delta$ and denote by $p,p_1,p_2$ any (smooth) $\delta$-solutions in order to keep our notations light. For the sake of simplicity we only consider the case $1<m\notin\mathbb{N}$, for which the asymptotic expansion at infinity $p=cx+q_1x^{1-\frac{1}{m}}+...$ holds (see remark \ref{rmk:m_noninteger} above).
\\

The main result of this section is
\begin{theo}
The $\delta$-solutions are unique up to finite $x$-translation.
\label{theo:uniqueness_delta-solutions}
\end{theo}
Let us start with some technical statements:
\begin{prop}
Any $\delta$-solution has an asymptotic expansion 
$$
p(x,y)=cx +x\left(q_{1}x^{-\frac{1}{m}}+...+q_{N}x^{-\frac{N}{m}}\right)+q^*+o(1)
$$
uniformly in $y$ when $x\rightarrow +\infty$, where $q_1...q_N,q^*\in\R$ and $1\leq N=[m]<m$.
\label{prop:uniqueness_delta-solutions}
\end{prop}
\begin{rmk}
The coefficients $q_i$, $q^*$ and the remainder $o(1)$ above may depend of course on $\delta$, which is fixed here.
\end{rmk}
\begin{proof}
We may proceed exactly as we did for the final viscosity solution, see section \ref{section:linear} and in particular the proof of Theorem \ref{theo:asymptotic_expansion}.
\end{proof}
The following holds at negative infinity, where we recall that $p(-\infty,y)=\delta>0$ uniformly in $y$.
\begin{lem}
We have that
$$
|\nabla p|\rightarrow 0, \quad |D^2 p|\rightarrow 0
$$ 
uniformly in $y$ when $x\rightarrow -\infty$.
\label{lem:pdelta->delta_C2}
\end{lem}
\begin{proof}
Let $w:=\frac{m^2}{m+1}p^{\frac{m+1}{m}}$ and $f:=(c+\alpha)p^{\frac{1}{m}-1}p_x$; recall that the Poisson equation
$$
\Delta w =f
$$
holds in the whole cylinder. Taking advantage of $p(-\infty,y) =\delta>0$ ($w$ thus being uniformly bounded from above and away from zero) we may safely apply our previous interior elliptic regularity argument on $\Omega_n:=]-n,-n+1[\times\T$ ($n\in\mathbb{N}$) to show that
$$
||p||_{W^{3,q}(\Omega_n)}\leq C
$$
for some constant $C>0$ and $q>d=2$ both independent of $n$.
\par
Setting
$$
\Omega:=]0,1[\times\T,\qquad p^n(x,y):=p(x-n,y),
$$
the previous estimate reads
$$
||p^n||_{W^{3,q}(\Omega)}\leq C.
$$
By compactness $W^{3,q}(\Omega)\subset\subset \mathcal{C}^2(\overline{\Omega})$ we may extract a subsequence $p^{n_k}\rightarrow p^{\infty}\text{ in }\mathcal{C}^2(\overline{\Omega})$. Since $p(-\infty,y)=\delta$, the limit $p^{\infty}(x,y)=p(-\infty,y)=cst=\delta$ is unique: standard separation arguments show that the whole sequence converges
$$
p^n\rightarrow \delta\text{ in }\mathcal{C}^2(\overline{\Omega}).
$$
which immediately implies our statement.
\end{proof}
\begin{prop}
The coefficients $q_1...q_N$ in the asymptotic expansion are unique.
\label{prop:uniqueness_asymptotic_expansion}
\end{prop}
\begin{proof}
Let $p_1$  and $p_2$ be two different $\delta$-solutions, thus satisfying
\begin{eqnarray*}
p_1 & = &  cx +x\left(q_{1,1}x^{-\frac{1}{m}}+...+q_{1,N}x^{-\frac{N}{m}}\right)+q^*_1+o(1)\\
p_2 & = &  cx +x\left(q_{2,1}x^{-\frac{1}{m}}+...+q_{2,N}x^{-\frac{N}{m}}\right)+q^*_2+o(1)
\end{eqnarray*}
when $x\rightarrow+\infty$ for some constants $q_{i,k},q^*_i\in\R$, \;$i=1,2$, \;$k=1...N$ and $N=[m]$.
\par
Assume by contradiction that $q_{1,1}>q_{2,1}$: we will first slide $p_2$ far enough to the right so that $p_2<p_1$ on the whole cylinder. Slowly sliding $p_2$ back to the left we will obtain a contact point between $p_1$ and a translate of $p_2$, thus contradiction the classical Maximum Principle.
\par
Lemma \ref{lem:pdelta->delta_C2} allows us to pin $p_1$ such that, for $x\leq 0$, there holds
\begin{enumerate}
 \item $\delta\leq p_1(x,y)\leq p_1(0,y)\leq2\delta$,
\item $|\nabla p_1|$ and $|\Delta p_1|$ are small.
\end{enumerate}
This can be done suitably sliding, since $p_1\rightarrow\delta$, $|\nabla p_1|\rightarrow 0$ and $|\Delta p_1|\rightarrow 0$  when $x\rightarrow -\infty$. In this proof $p_1$ will be fixed once and for all, and we will only slide $p_2$ with respect to $p_1$ (for the sake of clarity $p_2$ denotes below any translation).
\begin{itemize}
 \item 
Since we assumed that $q_{1,1}>q_{2,1}$ we have
$$
[p_1-p_2](+\infty,y)=+\infty
$$
for any (finite) translation $p_2$. Using $\partial_xp_i>0$ we may therefore slide $p_2$ far enough to the right so that
$$
x\geq 0\quad \Rightarrow\quad p_1(x,y)>p_2(x,y).
$$
\par
We claim that, applying a suitable comparison principle, we may assume that $p_1>p_2$ also holds for $x<0$. In order to see this, define $z:=p_1-p_2$ and subtract the equation for $p_2$ from the equation for $p_1$ to obtain
\begin{equation}
\mathcal{L}[z]:=-mp_2\Delta z+\Big{[}(c+\alpha)z_x-(\nabla p_1+\nabla p_2)\cdot\nabla z\Big{]}-(m\Delta p_1)z=0.
\label{eq:Lz=0}
\end{equation}
Testing $\overline{z}(x):=e^{\lambda x}$ as a supersolution for some $\lambda>0$, an elementary computation leads to
$$
\mathcal{L}[\overline{z}]=e^{\lambda x}\left(-mp_2\lambda^2+(c+\alpha)\lambda-(\partial_x p_1+\partial_x p_2)\lambda -m\Delta p_1\right).
$$
Sliding $p_2$ far enough to the right we have, for $x\leq 0$, that $p_2\sim\delta$ and that $\partial_xp_2$ is negligible. Since we also pinned $|\nabla p_1|$ and $|\Delta p_1|$ to be small, the main contribution in the parenthesis of the right-hand side above comes from the first two terms. Choosing $\lambda>0$ small enough, it is clearly possible to satisfy
$$
-mp_2\lambda^2+(c+\alpha)\lambda\gtrapprox -m\delta\lambda^2+c_0\lambda>0
$$
(choose for example $\lambda=c_0/2m\delta$), and therefore
$$
x<0\quad \Rightarrow \quad \mathcal{L}[\overline{z}]>0.
$$
Setting $z:=w\overline{z}$, the new variable $w$ satisfies this time an elliptic equation
$$
\tilde{\mathcal{L}}[w]=0,
$$
where $\tilde{\mathcal{L}}$ is uniformly elliptic, has positive zero-th order coefficient $\mathcal{L}[\overline{z}]>0$, and therefore satisfies the Minimum Principle.
\par
One the right boundary $x=0$ we may assume that $p_1(0,y)>p_2(0,y)$ (once again sliding $p_2$ far enough to the right), and therefore $w(0,y)>0$. At negative infinity we had exponential convergence $|p_i(x,y)-\delta|\leq C e^{\frac{c_0}{\delta m} x}$, and we chose the supersolution $\overline{z}=e^{\lambda x}$ to decay slowly ($\lambda>0$ was chosen small enough, for example $\lambda=c_0/2m\delta$), hence $|z|=|p_1-p_2|\leq Ce^{\frac{c_0}{m\delta}x}\ll |\overline{z}|$ and $w(-\infty,y)=0$. The Minimum Principle applied to $\tilde{L}[w]=0$ finally shows that $w(x,y)>0$ for $x<0$, and therefore $p_1>p_2$ on the whole cylinder $D=\R\times\T$ if $p_2$ is slided far enough to the right.
\item
Slowly sliding back to the left we obtain a first critical translation $p_2^*$, after which we cannot keep translating to the left without breaking $p_1\geq p_2$ (this critical translation exists because sliding $p_2$ far enough to the left the two solutions must cross at some point). By continuity we have that $z^*=p_1-p_2^*\geq 0$, and we claim that there exists a contact point $(x_0,y_0)\in D$ such that $z^*(x_0,y_0)=0$. Temporarily admitting this, we obtain a contradiction as follows: $z^*\geq 0$ satisfies \eqref{eq:Lz=0}, which is uniformly elliptic with bounded zero-th order coefficient, and attains a minimum point $z^*=0$ in $D$. The Minimum Principle shows that $z^*\equiv 0$, thus contradicting $q_{1,1}>q_{2,1}\Rightarrow z(+\infty,y)=+\infty$.
\par
In order to obtain such a contact point, assume by contradiction that $p_1>p_2^*$ on the whole cylinder: condition $q_{1,1}>q_{2,1}$ shows that $p_1-p_2^*\geq C_a>0$ on any sub-cylinder $x\geq -a$ for any $a>0$ large and some constant $C_a$. We may then slide $p_2$ slightly further to the left in such a way that $p_1>p_2$ if $x\geq a$. Repeating the above comparison argument for $x\leq a$, we see that $p_1>p_2$ also holds to the left, hence on the whole cylinder. This contradicts the fact that $p_2^*$ was a critical translation.
\end{itemize}
We just proved that $q_{1,1}>q_{2,1}$ cannot hold, and by symmetry $p_1\leftrightarrow p_2$ we obtain
$$
q_{1,1}=q_{2,1}.
$$
We may now repeat the very same argument to show that $q_{1,2}=q_{2,2}$, and so forth ($q_{1,k}=q_{2,k}\Rightarrow q_{1,k+1}=q_{2,k+1}$).
\end{proof}
We can now prove uniqueness of the $\delta$-solutions:
\begin{proof}
(of Theorem \ref{theo:uniqueness_delta-solutions}) Let $p_1,p_2$ be two $\delta$-solutions; we pin as before $p_1$ once and for all, and only slide $p_2$. Let us stress that both solutions have now the same coefficients $q_1...q_N$ in their asymptotic expansion at infinity
$$
i=1,2\quad x\rightarrow+\infty:\qquad p_i(x,y)=cx +x\left(q_1x^{-\frac{1}{m}}+...+q_Nx^{-\frac{N}{m}}\right)+q^*_i+o_i(1),
$$
except maybe for the last two terms (the lower order $q^*_i+o_i(1)$). We recall that $z:=p_1-p_2$ satisfies \eqref{eq:Lz=0} of the form $L[z]=0$, and remark the following: for any $\tau$-translation $p_2(x-\tau,y)$, uniqueness of the coefficients $q_1...q_N$ shows that
\begin{equation}
p_1(x,y)-p_2(x-\tau,y)\underset{+\infty}{=}c\tau +q^*_1-q^*_2+o(1).
\label{eq:p1-p2_at_infty}
\end{equation}
This means that, depending on the translation, only two scenarios are possible at infinity: either $[p_1-p_2](+\infty,y)=0$, either $[p_1-p_2](+\infty,y)=cst\neq 0$.
\par
We showed previously that sliding $p_2$ far enough to the right $p_1>p_2$ must hold, and that slowly sliding back to the left there exists a first critical translation $\overline{p}_2$ such that $p_1\geq \overline{p}_2$. Similarly translating $p_2$ far enough to the left we have that $p_1<p_2$, and there exists a first critical translation $\underline{p}_2$ coming from the left side such that $p_1\leq \underline{p}_2$.
\begin{enumerate}
 \item If there exists a contact point $p_1(x_0,y_0)=\overline{p}_2(x_0,y_0)$ then $\overline{z}=p_1-\overline{p}_2$ is nonnegative (because $\overline{p}_2$ is a critical translation coming from the right side), satisfies an elliptic equation $\mathcal{L}[\overline{z}]=0$ with bounded zero-th order coefficient, and attains an interior minimum point $\overline{z}(x_0,y_0)=0$. The classical Minimum Principle shows that $\overline{z}\equiv 0$, meaning precisely that $p_1$ can be deduced from $p_2$ by translation.
\\
We may therefore assume that no such contact point exists, and the only possible scenario is therefore that $[p_1-\overline{p}_2](+\infty,y)=0$ (otherwise $[p_1-\overline{p}_2](+\infty,y)=cst>0$ according to \eqref{eq:p1-p2_at_infty} and we could slide $p_2$ a little further to the left as in the proof of Proposition \ref{prop:uniqueness_asymptotic_expansion}, thus contradicting the fact that $\overline{p}_2$ is critical).
\item
Similarly arguing for $\underline{z}$, we may assume that $[p_1-\underline{p}_2](+\infty,y)=0$.
\item 
As a consequence $[\overline{p}_2-\underline{p}_2](+\infty,y)=0$, and therefore $\overline{p}_2=\underline{p}_2$. We conclude recalling that we constructed $\overline{p}_2\leq p_1\leq\underline{p}_2$, hence $p_1=\overline{p}_2=\underline{p}_2$.
\end{enumerate}
\end{proof}

\section*{Acknowledgment}
LM would like to thank the PREFERED French ANR project and NSF grant  DMS-0908011 for financial support, Penn State University and Stanford University for their hospitality. AN was supported by the NSF grant DMS-0908011.
\bibliographystyle{siam}
\bibliography{biblio}
\end{document}